\documentclass[12pt]{amsart}%
\usepackage{amsfonts}
\usepackage{amsmath}
\usepackage{amssymb}
\usepackage{graphicx}%

\makeatletter
\@addtoreset{equation}{section}
\makeatother

\newtheorem{theorem}{Theorem}[section]

\newtheorem{corollary}[theorem]{Corollary}

\newtheorem{lemma}[theorem]{Lemma}

\newtheorem{remark}[theorem]{Remark}

\makeatletter
\def\@makefnmark{}
\makeatother
\begin{document}
\title[Ground states of bi-harmonic equations with critical exponential growth]{Existence and Non-existence of Ground states of bi-harmonic equations involving constant and degenerate Rabinowitz potentials}
\author{Lu Chen, Guozhen Lu and Maochun Zhu}
\address{School of Mathematics and Statistics, Beijing Institute of Technology, Beijing 100081, P. R. China}
\email{chenlu5818804@163.com}
\address{Department of Mathematics\\
University of Connecticut\\
Storrs, CT 06269, USA}
\email{guozhen.lu@uconn.edu}
\address{Faculty of Science\\
Jiangsu University\\
Zhenjiang, 212013, P. R. China\\}
\email{zhumaochun2006@126.com}

\thanks{The first author was  supported partly by the National Natural Science Foundation of China (No. 11901031). The second author was supported partly by the Simons Foundation. The
third author was supported partly by Natural Science Foundation of China (12071185). }

\begin{abstract}
Recently, the authors of the current paper established in \cite{chenlu2}
the existence of a ground-state solution to the following bi-harmonic equation with the constant potential or Rabinowitz potential:
\begin{equation}\label{bi-harmonic1a}
(-\Delta)^{2}u+V(x)u=f(u)\ \text{in}\ \mathbb{R}^{4},
\end{equation}
when the nonlinearity has the special form $f(t)=t(\exp(t^2)-1)$ and $V(x)\geq c>0$ is a constant or the Rabinowitz potential. One of the crucial elements used in \cite{chenlu2} is  the Fourier rearrangement argument. However, this argument is not applicable if $f(t)$ is not an odd function.  Thus, it still remains open whether the equation (\ref{bi-harmonic1a}) with the general critical exponential nonlinearity $f(u)$ admits a ground-state solution even when $V(x)$ is a positive constant.

The first purpose of this paper is to develop a Fourier rearrangement-free approach to solve the above problem. More precisely, we will prove that there is a threshold $\gamma^{*}$ such that for any $\gamma\in (0,\gamma^*)$, the equation (\ref{bi-harmonic1a}) with the constant potential $V(x)=\gamma>0$ admits a ground-state solution, while does not admit any ground-state solution for any $\gamma\in (\gamma^{*},+\infty)$. The second purpose of this paper is to establish the existence of a ground-state solution to the  equation (\ref{bi-harmonic1a}) with any degenerate Rabinowitz potential $V$ vanishing on some bounded open set. Among other techniques, the proof also relies on a critical Adams inequality involving the degenerate potential which is of its own interest.
\end{abstract}

\maketitle {\small {\bf Keywords:} Rabinowitz potential, Ground state
solutions; Bi-harmonic equations;  Adams' inequalities, Nehari manifold. \\

}

\section{\bigskip Introduction}

We begin with considering the following nonlinear partial differential equation

\begin{equation}
\left(  -\Delta\right)  ^{m}u+V\left(  x\right)  u=f\left(  u\right)  \text{
in }%
\mathbb{R}
^{n},\label{potential eq}%
\end{equation}
where $m$ is a positive integer,  $V(x)$ is some nonnegative potential. Equations (\ref{potential eq}) with subcritical and critical growth have been extensively studied by many
authors. In the case $n>2m$, the subcritical and critical growth means that the
nonlinearity cannot exceed the polynomial growth of degree $\frac{n+2m}%
{n-2m}$ by the Sobolev embedding. While in the case $n=2m$, we say that
$f\left(  s\right)  $ has \textit{critical exponential growth} at infinity if
there exists $\alpha_{0}>0$ such that \
\begin{equation}
\underset{|t|\rightarrow+\infty}{\lim}\frac{|f\left(  t\right)|  }{\exp\left(
\alpha t^{2}\right)  }=%
\genfrac{\{}{.}{0pt}{}{0\text{, \ \ for }\alpha>\alpha_{0}}{+\infty\text{, for
}\alpha<\alpha_{0}}.
\label{exponential critical1}%
\end{equation}

The critical exponential growth in the case $m=1,n=2$ is consistent with the Trudinger-Moser inequality (\cite{Mo}, \cite{Tru}), while in the case $m=2,n=4$ is given by the Adams inequality \cite{A}. The study of the existence for equation (\ref{potential eq}) with the critical exponential growth involves a lack of compactness, a.e. the Palais-Smale compactness condition may fail at some special level. However, unlike the equations on bounded domain (see  e.g., \cite{Ambrosetti}, \cite{Brezis}, \cite{de Figueiredo},    \cite{LamLu-ANS}, \cite{LamLu-JGA}, \cite{Lam}), the loss of compactness for equations (\ref{potential eq}) may be produced not only by the concentration phenomena but also by the vanishing phenomena.

\medskip

 The earlier study of the existence of solutions for equation (\ref{potential eq}) with the critical exponential growth can date back to the work of Atkinson and Peletier \cite{Atkinson, Atkinson2}. Indeed, the authors obtained the existence of ground state solutions for equation (\ref{potential eq}) in $\mathbb{R}^2$ by assuming that there exists some $y_{0}>0$ such that
$g\left(  t\right)  =\log f\left(  t\right)  $ satisfies $$g^{\prime}\left(
t\right)  >0,g^{\prime\prime}\left(  t\right)  \geq0,$$ for any $t\geq y_{0}$. This kind of growth condition allows us to take the nonlinearity $f\left(
t\right)  =\left(  t^{2}-t\right)  \exp\left(  t^{2}\right), $ which has critical exponential growth.

\medskip

As far as we are concerned, having a positive lower bound has become a standard assumption on the potential $V(x)$ in dealing with the existence of solutions to the equations (\ref{potential eq}) in the literature, we will briefly describe some of the relevant works below.

When $V\left(  x\right)  $ is a coercive potential, that is,
\[
V\left(  x\right)  \geq V_{0}>0,\text{and additionally either}\underset{x\rightarrow\infty
}{\lim}V(x)=+\infty\text{ or }\frac{1}{V}\in L^{1}\left(
\mathbb{R}
^{n}\right)  ,
\]
the existence results of equation (\ref{potential eq}) can be
found in the papers e.g., \cite{doO1}, \cite{Lamlu}, \cite{Yang}, \cite{zhao} and the
references\ therein. Their proofs depend crucially on the compact embeddings
given by the coercive potential, and the vanishing phenomena can be ruled out.
\vskip0.1cm

When $V\left(  x\right)  $ is the constant potential, i.e. $V\left(  x\right)
=\gamma>0$, the natural space for a variational treatment of (\ref{potential eq})
is $W^{m,2}\left(
\mathbb{R}
^{n}\right)  $. It is well known that the embedding $W^{m,2}\left(
\mathbb{R}
^{n}\right)  \hookrightarrow L^{2}\left(
\mathbb{R}
^{n}\right)  $\ is continuous but not compact, even in the radial case.

In the case $m=1$ and $n=2$, the authors of  \cite{Alves}  obtained
the existence of a ground solution to equation
(\ref{potential eq})   under the assumptions that for any
$p>2$,
\begin{equation}
f\left(  s\right)  \geq\eta_{p}s^{p-1},\forall s\geq0, \label{polo}%
\end{equation}
where $\eta_{p}$ is some constant depending on $p$. In \cite{Rufs}, the authors also obtained the existence of a ground state solution to (\ref{potential eq}) under
 \begin{equation}
\underset{|s|\rightarrow\infty}{\lim}\frac{sf\left(  s\right)  }{\exp\left(
32\pi^{2}s^{2}\right)  }\geq\beta_{0}>0. \label{exp}%
\end{equation}
In general, \eqref{polo} and \eqref{exp} are not comparable.  In \cite{IMN}, the authors proved that there exists a positive $\gamma^*$ such that for any $0<V=\gamma<\gamma^*$, the equation (\ref{potential eq})  has a ground state solution under a weaker assumptions than both \eqref{polo} and \eqref{exp}.

In
the case $m=2$, the existence of a nontrivial solution to equation
(\ref{potential eq}) was obtained   in \cite{chenluzhang}  under the assumption that  \eqref{polo} holds (see also
\cite{bao}),  and  in \cite{sani}
  under the assumption that \eqref{exp} holds. The existence of a nontrivial solution to (\ref{potential eq}) under the assumption weaker than  both \eqref{polo} and \eqref{exp} was established in \cite{chenlu2}. Furthermore, the existence of a ground state solution to (\ref{potential eq})
  was only recently proved  in \cite{chenlu2}. (see more detailed discussions below.)

 \medskip

 We recall that the following Trudinger-Moser inequality holds (see \cite{ruf}, \cite{liruf}):

 \begin{equation}\label{trudinger-morser}
\sup_{u\in W^{1,n}(\mathbb{R}^{n}), \int_{\mathbb{R}^n}\left(|\nabla u|^n+|u|^n\right)dx\leq1}\int_{\mathbb{R}^{n}}\Phi_n\big(\alpha_n |u|^{\frac{n}{n-1}}\big)dx<\infty.
\end{equation}
where $\omega _{n - 1}$ denotes the area of the unit sphere in $\mathbb{R}^n$ and $\Phi_{n}(t):=e^{t}-\sum_{i=0}^{n-2}\frac{t^{i}}{i!}$.
The proof of the Trudinger-Moser inequality in \cite{ruf} and \cite{liruf}
relies on the P\'{o}lya-Szeg\"{o} inequality and the symmetrization argument. Subsequently, the authors in \cite{LaLu4} used a symmetrization-free approach to give a simple proof for the sharp Trudinger-Moser inequalities in $W^{1,n}(\mathbb{R}^n)$ (see also \cite{LamLu-AIM}).
 We have proved more recently in \cite{chenlu1} and \cite{chenlu3}  the following Trudinger-Moser inequality in higher dimension $\mathbb{R}^n$  ($n\ge 2$) under the less restrictive constraint $$\int_{\mathbb{R}^n}\left(|\nabla u|^n+V(x)|u|^n\right)dx\le 1,$$ where $V(x)\geq 0$ satisfying:
\vskip0.1cm

(V1): $V(x)=0$ at $B_{\delta}(0)$ and $V(x)\geq c_0$ in $\mathbb{R}^n\setminus B_{2\delta}(0)$ for some $c_0, \delta>0$.

\vskip0.1cm

{\bf Theorem A.}
Assume that the potential $V(x)$ satisfies the condition (V1). Then
\begin{equation}\label{trudinger-morser}
\sup_{u\in W^{1,n}(\mathbb{R}^{n}), \int_{\mathbb{R}^n}\left(|\nabla u|^n+V(x)|u|^n\right)dx\leq1}\int_{\mathbb{R}^{n}}\Phi_n\big(\alpha_n |u|^{\frac{n}{n-1}}\big)dx<\infty.
\end{equation}

\vskip0.1cm

We note that  the loss of a positive lower bound of the potential $V(x)$ makes this inequality become fairly nontrivial.

Sharp Adams inequalities on the entire space $\mathbb{R}^n$ were studied in ~\cite{RS}
under the constraint
$$\{
u\in W^{{m},\frac{{n}}{m}}| \Vert(I-\Delta)^{\frac{m}{2}}u\Vert_{\frac{n}{m}%
}\leq1
\},$$
when $m$ is an even integer. When the order $m$ of the derivatives is odd, a sharp Adams inequality was established
in \cite{LamLu-JDE2012}. The same authors in \cite{LaLu4} give a unified approach for all orders $m$ of derivatives including fractional orders of derivatives through the rearrangement-free argument.  Furthermore, they also obtained
the following sharp Adams inequality under the Sobolev norm constraint: let $\tau>0$,

\begin{equation}
\underset{\Vert\Delta u\Vert_{2}^{2}+\tau\Vert u\Vert_{2}^{2}\leq{1}}%
{\underset{u\in W{^{2,2}(\mathbb{R}^{4})}}{\sup}}\int_{{\mathbb{R}^{4}}}\left(
\exp(\beta|u(x)|^{2})-1\right)  dx\left\{
\begin{array}
[c]{l}%
\leq C\text{ if }\beta\leq{32\pi}^{2}{,}\\
=+\infty\text{ \ if }\beta>{32\pi}^{2}{.}%
\end{array}
\right.  \label{Adams entire space}%
\end{equation}

As an application of critical Adams inequality  \eqref{Adams entire space} on the whole space $\mathbb{R}^4$, the authors of \cite{chenlu2} obtained the existence of a non-trivial radial solution to the following bi-harmonic equation with the constant
potential:
\begin{equation}
(-\Delta)^{2}u+\gamma u=f(u)\ \text{in}~\mathbb{R}^{4},\ \label{bi-harmonic11}%
\end{equation}
when the nonlinearity $f(t)$ has the critical exponential growth at infinity. However, the existence of a ground state solution was not proved in \cite{chenlu2}. More precisely, in \cite{chenlu2} the following was proved:

\medskip
{\bf Theorem B.} \cite{chenlu2}
\label{thm3} Assume that $f$ satisfies $f(0)=0$ and the conditions
 (i), (ii), (iii) and (iv) in Section 2,  then there exists $\gamma^{\ast
}\in(0,+\infty]$ such that for any $\gamma\in(0,\gamma^{\ast})$, the equation
(\ref{bi-harmonic11}) admits a non-trivial radial solution. Moreover,
$\gamma^{\ast}$ is equal to the radial Adams' ratio:
\[
C_{A}^{\ast}=\sup\left\{  \frac{2}{\Vert u\Vert_{2}^{2}}\int_{\mathbb{R}^{4}%
}F(u)dx|\ u\in W_{r}^{2,2}(\mathbb{R}^{4})\setminus \{0\},\Vert\Delta u\Vert_{2}^{2}\leq
\frac{32\pi^{2}}{\alpha_{0}}\right\}  ,
\]
where $W_{r}^{2,2}(\mathbb{R}^{4})$  is the collection of all radial functions
in $W^{2,2}(\mathbb{R}^{4})$ and $F(t)=\int_{0}^{t}f(s)ds$. In particular, $\gamma^{\ast}=+\infty$ is equivalent to
\[
\lim_{|t|\rightarrow\infty}\frac{t^{2}F(t)}{\exp(\alpha_{0}t^{2})}=+\infty.
\]

Furthermore, if the nonlinearity has the special form $f(t)=\lambda t\exp(2|t|^{2})$, the authors  can further prove that the solutions obtained are ground-state solutions:

\medskip

{\bf Theorem C.}  \cite{chenlu2} For any $\gamma\in(0,+\infty)$, the equation

\[
(-\Delta)^{2}u+\gamma u=\lambda u\exp(2|u|^{2})\ \text{in}~\mathbb{R}^{4}%
\]
admits a radial ground state solution if $\lambda\in\left(  0,\gamma\right)  $.

\begin{remark}We cannot use the Schwarz symmetrization principle directly in the proof of Theorem C due to the presence of the higher order derivatives. In order to overcome this difficulty, in \cite{chenlu2} the authors applied the Fourier rearrangement proved by Lenzmann and Sok in \cite{Lenzmann} to obtain a radially minimizing sequence for the
infimum on the Pohozaev manifold. We stress that the Fourier rearrangement argument requires that $f(s)$ must be odd and all the
coefficients of the Taylor series for the primitive function $F(s)$ must be positive.
\end{remark}
Based on the above result, by exploiting the
relationship between the Nehari manifold and the corresponding limiting Nehari manifold, the authors can also obtain the existence of ground state
solutions of the bi-harmonic equation with the non-radial Rabinowitz type potential introduced in \cite{Rabinowitz1}:

{\bf Theorem D.} \cite{chenlu2}
\bigskip\label{thm6} Assume that $V\left(  x\right)  $ is a continuous
function satisfying
\[
0<\lambda<V_{0}=\underset{x\in\mathbb{R}^{4}}{\inf}V\left(  x\right)<\underset{x\in\mathbb{R}^{4}}{\sup}V\left(  x\right)
=\underset{\left\vert x\right\vert \rightarrow\infty}{\lim}V\left(  x\right)
=\gamma<+\infty,
\]
the equation \[
(-\Delta)^{2}u+V(x) u=\lambda u\exp(2|u|^{2})\ \text{in}~\mathbb{R}^{4}%
\]
 admits a ground state solution which is not necessarily radial.

\medskip

The following remarks are in order. First, as we pointed out earlier, the method  in \cite{chenlu2} of using the Fourier rearrangement to establish the existence of a ground state solution is not applicable to more general nonlinearity $f$ than the special form $f(t)=\lambda t\exp(2|t|^{2})$ (see Theorem C).
Therefore,   new method without using the Fourier rearrangement needs to be developed to deal with the more general nonlinearity $f$. Our method in this paper does not rely on the P\'olya-Szeg\"{o}
inequality nor the Fourier rearrangement.
 Second, in our earlier work \cite{chenlu2} we assume that the potential $V$  has a positive low bound in the entire space $\mathbb{R}^4$. Another novelty in this paper is that the potential $V$
can be degenerate on an open bounded set in $\mathbb{R}^4$. Third, the Adams inequality under the Sobolev norm associated with the degenerate potential $V$ is established and is of its independent interest. This Adams embedding is necessary to establish the existence of the ground state solution. Fourth, our result is fairly sharp in the sense that we have identified a threshold $\gamma^*$ for the bi-harmonic equation with constant potential $(-\Delta)^{2}u+\gamma u=f(u)\ \text{in}~\mathbb{R}^{4}$ such that the existence of a ground state solution is guaranteed for any $\gamma\in (0, \gamma^*)$ and the nonexistence of any ground state solution is assured for any $\gamma\in (\gamma^*, \infty)$.

\medskip
Therefore,
the main purpose of this paper is to answer  the following two questions:

   \medskip
   1. Can the solution in Theorem B be a ground-state solution and  does Theorem C still hold when the nonlinearity $f$ is a more general function satisfying the critical exponential growth and the Ambrosetti-Rabinowitz condition  rather than having a special form $f(t)=\lambda t\exp(2|t|^{2})$?

\medskip

2. Does Theorem D still hold when the potential $V$ is a degenerate Rabinowitz type potential and furthermore the nonlinearity $f$ is a more general function satisfying the critical exponential growth and the Ambrosetti-Rabinowitz condition  rather than having a special form $f(t)=\lambda t\exp(2|t|^{2})$?

\section{The main results}
 Motivated by the results just described, in this paper, we first consider the following bi-harmonic equation with the constant
potential:
\begin{equation}
(-\Delta)^{2}u+\gamma u=f(u)\ \text{in}~\mathbb{R}^{4},\ \label{3.1}%
\end{equation}
where the nonlinearity $f(t)$ satisfies $f(0)=0$  and  the following properties:

(i) has critical exponential growth (\ref{exponential critical1}).

(ii) (Ambrosetti-Rabinowitz condition)(A-R) \cite{Ambrosetti, Rabinowitz2, Rabinowitz3}  There exists $\mu>2$ such that $0<\mu F(t)=\mu\int_{0}^{t}f(s)ds\leq tf(t)$
for any $t\in\mathbb{R}$;\newline

(iii) there exist $t_{0}$ and $M_{0}>0$ such that $F(t)\leq M_{0}|f(t)|$ for any
$|t|\geq t_{0}$.

(iv) $f(t)=o(t)$ as $t\rightarrow 0$.

(v) $f(t)\in C^1$ and $\frac{f(t)}{t}$ is increasing.

\begin{remark}\label{adrem}
The condition (ii) implies that $F(t)=o(t^2)$ as $t\rightarrow 0$. Indeed, the condition (ii) implies that $\left(\frac{F(t)}{t^{\mu}}\right)^{\prime}>0$, from which one can immediately get $F(t)=o(t^2)$ as $t\rightarrow 0$. From  conditions (i),(ii) and (iv) above, one can obtain the following growth condition for $f(t)$: for any $\varepsilon>0$ and $\beta_{0}>\alpha_{0}$, there
exists $C_{\varepsilon}$ such that
\begin{equation}
|f(t)|\leq\varepsilon|t|+C_{\varepsilon}|t|^{\mu}\left(e^{\beta_{0}t^{2}}-1\right),\forall
\ t\in\mathbb{R}\label{xin1}.%
\end{equation}
From (v), one can also easily check that the function $f(t)t-2F(t)$ is increasing.
\end{remark}

Our first result is the following

\begin{theorem}\label{thm1}
\label{thm3} Assume that $f$ satisfies $f(0)=0$ and the conditions
 (i), (ii) and (iii), then there exists $\gamma^{\ast
}\in(0,+\infty]$ such that the equation
(\ref{3.1}) admits a ground-state solution for any $\gamma\in(0,\gamma^{\ast})$, and does not admit any ground-state solution for any
$\gamma\in(\gamma^{\ast},+\infty)$, where $\gamma^*$
 is equal to the Adams' ratio:
\[
\gamma^{\ast}=\sup\left\{  \frac{2}{\Vert u\Vert_{2}^{2}}\int_{\mathbb{R}^{4}%
}F(u)dx|\ u\in W^{2,2}(\mathbb{R}^{4})\setminus \{0\},\Vert\Delta u\Vert_{2}^{2}\leq
\frac{32\pi^{2}}{\alpha_{0}}\right\}
.\]
In particular, $\gamma^{\ast}=+\infty$ is equivalent to
\[
\lim_{|t|\rightarrow\infty}\frac{t^{2}F(t)}{\exp(\alpha_{0}t^{2})}=+\infty.
\]
\end{theorem}

The above theorem reveals an interesting relation between an Adams type inequality and the nonexistence of a ground-state solution of bi-harmonic equation with the critical exponential growth. In fact, as an immediate consequence of Theorem \ref{thm1}, we can conclude the following
\begin{corollary}Assume that $f$ satisfies $f(0)=0$ and the conditions
 (i), (ii), (iii) and (iv)  and $F(t)=\int_{0}^{t}f(s)ds$. Then
 the following Adams type inequality $$\sup_{\|\Delta u\|_2\le \frac{32\pi^2}{\alpha_0}}\frac{\int_{\mathbb{R}^{4}}F(u)dx}{\int_{\mathbb{R}^{4}}|u(x)|^{2}dx}
\leq C $$ holds for some $0<C<\infty$ if and only if there exists some $\beta^*\in \mathbb{R}$ such that for $\beta>\beta^*$,
$(-\Delta)^2 u+\beta u=f(u)$ does not admit any ground-state solution.
\end{corollary}

As we mentioned before, the loss of compactness for equations (\ref{potential eq}) may be produced not only by the concentration phenomena but also by the vanishing phenomena. In the literature, in order to exclude the vanishing phenomena, one can introduce the coercive potential (see \cite{zhao,Yang,Lamlu}), or apply some symmetrization  argument (see \cite{chenlu2,chenluzhang,bao,Alves, IMN, Rufs,MS2}). However, for our bi-harmonic equation (\ref{3.1}), the symmetrization argument fails, since the nonlinearity $f(t)$ in Theorem \ref{thm1} needn't be an odd function. Hence, neither the Schwarz symmetrization principle  nor the Fourier rearrangement used in \cite{chenlu2} can be applied to prove Theorem \ref{thm1}. For this reason, we will explore the relationship between the Nehari manifold and Pohozaev manifold, and develop a rearrangement-free approach to exclude the vanishing phenomena (see Lemma \ref{lem3.4} in Section 3). This rearrangement-free approach has it's own interests and can be used in the settings where the symmetrization technique does not work.
\vskip0.1cm

In the recent work \cite{chenlu1}, the authors of this paper established the existence of ground-state solution for the following  Schrodinger equation involving the degenerate Rabinowitz potential:
\begin{equation*}
-\Delta u+V(x)u=f(u)\ \text{in}\ \mathbb{R}^{2},
\end{equation*}
where $V(x)\geq 0$ and may vanishes on an open set of $\mathbb{R}^2$ and $f$ has the critical exponential growth. This is the first existence result for elliptic equation involving critical exponential growth without standard potential assumption: having the positive lower bound.
More recently, the authors established in \cite{chenlu3} the existence of ground state solutions to the following quasilinear Schr\"{o}dinger equation with the degenerate potential $V$:\begin{equation}\label{con}\begin{cases}
-\text{div}(|\nabla u|^{n-2}\nabla u) +V(x)|u|^{n-2}u=f(u)\ \text{in}\ \mathbb{R}^{n},\\
u\in W^{1,n}(\mathbb{R}^n).
\end{cases}\end{equation}

Motivated by the works in \cite{chenlu1} and \cite{chenlu3}, we are interested to study the existence of  ground state
solutions to the following bi-harmonic  equation:
\begin{equation}\label{degen}
\left(  -\Delta\right)  ^{2}u+V(x)u=f(u),\ \ x\in \mathbb{R}^4,%
\end{equation}
where $f(t)$ satisfies $f(0)=0$ and the conditions (i)-(v), and
the potential $V\left(  x\right)\geq 0$ satisfies
\vskip0.1cm

(V1) $V(x)=0$ at $B_{\delta}(0)$ and $V(x)\geq c_0$ in $\mathbb{R}^4\setminus B_{2\delta}(0)$ for some $c_0, \delta>0$,
\vskip0.1cm

(V2)
\[
\sup_{x\in\mathbb{R}^{4}}V(x)=\lim_{|x|\rightarrow\infty}V(x)=V_{\infty}>0.
\]

To this end, we first need to establish the following sharp critical Adams inequality involving the degenerate potential $V(x)$.

\begin{theorem}\label{Ad}
Assume that the potential $V(x)$ satisfies the condition $(V1)$, then
\begin{equation}\label{AdamsInequality}
\sup_{u\in W^{2,2}(\mathbb{R}^{4}), \int_{\mathbb{R}^4}(|\Delta u|^2+V(x)u^2)dx\leq1}\int_{\mathbb{R}^{4}}\big(e^{32\pi^2 u^2}-1\big)dx<\infty.
\end{equation}
\end{theorem}
\begin{remark}
It should be noted that the loss of a positive lower bound of the potential $V(x)$ makes this problem become fairly complicated and classical methods such as symmetrization argument and  blow-up analysis fail in dealing with this problem. Furthermore, because $\|\Delta (|u|)\|_{L^{2}(\mathbb{R}^4)}\leq \|\Delta (u)\|_{L^{2}(\mathbb{R}^4)}$ does not hold, it is not sufficient to only prove that this inequality \eqref{AdamsInequality}  holds for all positive functions in Sobolev space $W^{2,2}(\mathbb{R}^4)$. We will improve classical rearrangement-free argument developed by Lam and Lu in \cite{LamLu-AIM,LamLu-JDE2012} to overcome this difficulty.
\end{remark}

Based on Theorems \ref{thm1} and \ref{Ad}, by exploring the relationship between the Nehari manifold and the corresponding limiting Nehari manifold, we can obtain the following result.

\begin{theorem}\label{thm3}
 Assume that $V\left(  x\right)\geq 0$ is a continuous
function satisfying the condition $(V1)$ and $(V2)$, then for any  $V_{\infty}\in(0,\gamma^{\ast})$, the equation
\eqref{degen} admits a ground-state solution, where $\gamma^{\ast}$ is defined as in Theorem \ref{thm1}.
\end{theorem}

This paper is organized as follows. Section 3 is devoted to the proofs of the existence and nonexistence of a ground-state solution to bi-harmonic equation \eqref{3.1} for general critical exponential growth.
In Section 4, we will prove the crtitical Adams inequalities involving degenerate potential.
In Section 5, we will prove the existence of ground state solutions for the
bi-harmonic equation \eqref{degen} with the degenerate Rabinowitz type potential.

Throughout this paper, the letter $c$ always denotes some positive constant
which may vary from line to line.

\section{Existence and nonexistence of a ground-state solution of bi-harmonic equations with constant potentials: Proof of Theorem \ref{thm1}}

In this section, we are concerned with the ground states of the following quasilinear bi-harmonic equation (\ref{3.1}) with the constant potential and the nonlinearity $f(t)$ satisfying (i)-(iii). Namely, we will prove Theorem \ref{thm1}.

\medskip
The associated functional and Nehari Manifold are
$$I_{\gamma}(u)=\frac{1}{2}\int_{\mathbb{R}^{4}}\left(  \left \vert \Delta
u\right\vert ^{2}+\gamma|u|^2\right)
dx-\int_{\mathbb{R}^{4}}F(u)dx,\ \
\mathcal{N}_{\gamma}=\left\{  \left.  u\in W^{2,2}\left(  \mathbb{R}^{4}\right)
\right\vert u\neq0,N_\gamma\left(  u\right)  =0\right\}$$
respectively, where
$$N_\gamma\left(  u\right)  =\int_{\mathbb{R}^{4}}\left(  \left\vert \Delta
u\right\vert ^{2}+\gamma|u|^2\right)
dx-\int_{\mathbb{R}^{4}}f(u)udx.$$
One can easily verify that if $u\in\mathcal{N}_{\gamma}$, then
$$I_\gamma(u)=\frac{1}{2}\int_{\mathbb{R}^4}(f(u)u-2F(u))dx.$$
In the following, we will denote the Sobolev norms by
$$\|u\|_{W_\gamma^{2,2}(\mathbb{R}^4)}=\left(\int_{\mathbb{R}^{4}}\left(  \left \vert \Delta
u\right\vert ^{2}+\gamma|u|^2\right)dx\right)^{1/2}$$ and $$\|u\|_{W_V^{2,2}(\mathbb{R}^4)}=\left(\int_{\mathbb{R}^{4}}\left(  \left \vert \Delta
u\right\vert ^{2}+V(x)|u|^2\right)dx\right)^{1/2},$$
  respectively.

\medskip
We first claim

\begin{lemma}\label{6.1}
For any $u\in W^{2,2}\left(  \mathbb{R}^{4}\right) \setminus \{0\}$, there exists a
$t_{u}>0$ such that $t_{u}u\in\mathcal{N}_{\gamma}$.
\end{lemma}

\begin{proof}
For any $u\in W^{2,2}\left(  \mathbb{R}^{4}\right)$, we have
\begin{equation}\begin{split}
N_{\gamma}(tu)=t^2\int_{\mathbb{R}^4}(|\Delta u|^2+\gamma|u|^2)dx-\int_{\mathbb{R}^4}f(tu)(tu)dx.
\end{split}\end{equation}
Obviously, from the expression of $N(tu)$ and the conditions (i) and (iv), it is not hard to find that
$N(tu)<0$ for large $t$ and $N(tu)>0$ for small $t$.
Hence, there exists a  $t_{u}>0$ such that
$t_{u}u\in\mathcal{N}_\gamma$.
\end{proof}
\medskip
We recall that a solution $u$ of (\ref{3.1})
is called a ground state if $$I_{\gamma}(u)= \inf\{I_{\gamma}(w)\ |\ w \neq 0,\  w\ \mathrm{is\ a\ weak\ solution\ of}%
\ \text{(\ref{3.1})}\}.$$  Set $m_{\gamma}=\inf \{I_\gamma(u)\ |\ N_\gamma(u)=0\}$, then the existence of ground-state solution of equation \eqref{3.1} is equivalent to the attainability of $m_{\gamma}$. We claim
\begin{lemma}
\label{upper}There holds
\begin{equation}
0<m_{\gamma}\leq \frac{16\pi^2}{\alpha_0}. \label{the rang}%
\end{equation}

\end{lemma}

\begin{proof}
We first show that $m_{\gamma}>0$.  We prove this by contradiction. Assume that there
exists some sequence $u_{k}\in$ $\mathcal{N}_\gamma$ such that $I_\gamma\left(
u_{k}\right)  \rightarrow0$, that is,%
\[
\lim\limits_{k\rightarrow+\infty}\int_{\mathbb{R}^{4}}\big(f(u_{k}%
)u_{k}-2F(u_{k})\big)dx=0,
\]
which together with (A-R) condition and $u_{k}\in\mathcal{N}_\gamma$ yields that
\begin{equation}
\lim_{k\rightarrow+\infty}\int_{\mathbb{R}^{4}}\left(|\Delta u_{k}|^{2}%
+|u_{k}|^{2}\right)dx=0.\label{t1c}%
\end{equation}
On one hand, it follows from \eqref{xin1} and $u_{k}\in\mathcal{N}_\gamma$ that
\begin{equation}%
\begin{split}
1&=\int_{\mathbb{R}^{4}}f(u_{k})\frac{u_{k}}{\Vert u_{k}\Vert^2_{W_\gamma^{2,2}(\mathbb{R}^4)}}dx\\
&\leq\int_{\mathbb{R}^{4}}\big(\varepsilon\frac{|u_{k}|^{2}}{\Vert
u_{k}\Vert^2_{W_\gamma^{2,2}(\mathbb{R}^4)}}+C_{\varepsilon}\frac{|u_{k}|^{\mu+1}}{\Vert u_{k}\Vert_{W_\gamma^{2,2}(\mathbb{R}^4)}%
^{2}}\left(  e^{\beta_{0}u_{k}^{2}}-1\right)\big)dx
\end{split}
\label{t2}%
\end{equation}

On the other hand, by \eqref{t1c},  Adams inequality (\ref{Adams entire space}) and the fact
that $\mu>2$, we get for any $p>1,$
\begin{align*}
&\frac{1}{\Vert u_{k}\Vert^2_{W_\gamma^{2,2}(\mathbb{R}^4)}}\int_{\mathbb{R}^{4}}|u_{k}|^{2}%
|u_{k}|^{\mu-2}e^{\beta_{0}u_{k}^{2}}dx  \\ &\leq\frac{1}{\Vert u_{k}%
\Vert^2_{W_\gamma^{2,2}(\mathbb{R}^4)}}\left(  \int_{\mathbb{R}^{2}}|u_{k}|^{(\mu+1) p}dx\right)
^{1/p}\left(  \int_{\mathbb{R}^{4}}\left(  e^{p^{\prime}\beta_{0}u_{k}^{2}%
}-1\right)  dx\right)  ^{1/p^{\prime}}\\
& \leq c\Vert u_{k}\Vert_{W_\gamma^{2,2}(\mathbb{R}^4)}^{\mu-1}\rightarrow0,\text{as }k\rightarrow
\infty.
\end{align*}
which is a contradiction with \eqref{t2}.
\medskip

Next, we prove that $m_{\gamma}\leq \frac{16\pi^2}{\alpha_0}$. Let $w\in W^{2,2}(\mathbb{R}^4)$ such that $\int_{\mathbb{R}^{4}}\left(  \left\vert
\Delta w\right\vert ^{2}+\left\vert w\right\vert
^{2}\right)  dx=1.$
Then there exists $t_{w}>0$ such that $$\int_{\mathbb{R}^4}\left(|\Delta(t_{w} w)|^2+(t_{w} w)^2-f(t_{w} w)(t_{w} w)\right)dx=0,$$
which implies that
\begin{equation}\begin{split}
m_\gamma&\leq \frac{1}{2}\int_{\mathbb{R}^4}\left(|\Delta(t_{w} w)|^2+(t_{w} w)^2\right)dx
-\int_{\mathbb{R}^4}F(t_{w} w)\\
&<\frac{t_{w}^2}{2}\int_{\mathbb{R}^4}\left(|\Delta w|^2+|w|^2\right)dx=\frac{t_{w}^2}{2}.
\end{split}\end{equation}
On the other hand, $\frac{f(t w)}{t}w$ is monotone increasing about the variable $t$. Set $m_{\gamma}=\frac{t_\gamma^2}{2}$, then we derive that
\begin{equation}\begin{split}
\int_{\mathbb{R}^4}\frac{f(t_{\gamma}w)}{t_{\gamma}}wdx&\leq \int_{\mathbb{R}^4}\frac{f(t_{w} w)}{t_w}wdx\\
&=\int_{\mathbb{R}^4}|\Delta w|^2+\gamma|w|^2dx=1,
\end{split}\end{equation}
which implies that $$\sup_{\int_{\mathbb{R}^4}(|\Delta w|^2+ |w|^2dx)\leq 1}\int_{\mathbb{R}^4}\frac{f(t_{\gamma}w)}{t_{\gamma}}wdx<\infty.$$
Since $f$ has the critical exponential growth, we derive $m_\gamma=\frac{t_{\gamma}^2}{2}\leq \frac{16\pi^2}{\alpha_0}$ by the critical Adams inequality which was established  in \cite{LaLu4}.
\end{proof}
Now, we introduce the Adams ratios:
\[
C_{A}^{L}=\sup\{\frac{2}{\left\Vert u\right\Vert _{2}^{2}}\int_{\mathbb{R}%
^{4}}F(u)|\ u\in W^{2,2}(\mathbb{R}^{4})\setminus \{0\},\left\Vert \Delta u\right\Vert
_{2}^{2}\leq L\}.
\]
The Adams threshold $R(F)$ is given by
\[
R(F)=\sup\{L>0\ |\ C_{A}^{L}<+\infty\}.
\]
We denote by $\gamma^{*}=C_{A}^{R(F)}$ the ratio at the threshold\ $R(F)$.
It follows from the critical exponential growth of nonlinearity $f$ and the Adams inequality with the exact growth condition in $W^{2,2}(\mathbb{R}^4)$ (\cite{MS}) that $R(F)=32\pi^{2}/\alpha_{0}$.

Next, we claim that

\begin{lemma}\label{m}
If $\gamma<\gamma^{*}$, then $m_\gamma<\frac{16\pi^2}{\alpha_0}$.
\label{lemm2}
\end{lemma}

\begin{proof}
The proof is divided into two steps:
\medskip

\emph{Step 1:}  Define the Pohozaev manifold $\mathcal{P}_\gamma$ by
$$\mathcal{P}_\gamma=\{u\in W^{2,2}(\mathbb{R}^4)|\ \int_{\mathbb{R}^4}\gamma |u|^2dx=2\int_{\mathbb{R}^4}F(u)dx\},$$ and
 $M_{p}=\inf\left\{  I_{\gamma}\left(  u\right)  ,u\in\mathcal{P}_{\gamma
}\right\}$, we claim $m_\gamma\leq M_p$.
\vskip0.1cm

Assume that $\{u_k\}_k$ is a minimizing sequence for $M_{p}$, that is $u_k\in \mathcal{P}_\gamma$ and $\lim\limits_{k\rightarrow +\infty}I(u_k)=\lim\limits_{k\rightarrow +\infty}\frac{1}{2}\|\Delta u_k\|_2^2=M_{p}$. Choosing $\lambda_k$ such that
$$\int_{\mathbb{R}^4}|\Delta (u_k(\lambda_k x))|^2dx+\gamma\int_{\mathbb{R}^4}|u_k(\lambda_kx)|^2dx=\int_{\mathbb{R}^4}f(u_k(\lambda_k x))u(\lambda_k x)dx.$$
Direct computations yields $$\lambda_k^4=\frac{\int_{\mathbb{R}^4}\big(f(u_k)u_k-2F(u_k)\big)dx}{\|\Delta u_k\|_2^2},$$ which together with (A-R) condition gives that $\lambda_k>0$. Obviously $u_k(\lambda_kx)\in \mathcal{P}_\gamma\cap \mathcal{N}_\gamma$ and
\begin{equation}
m_\gamma\leq \lim\limits_{k\rightarrow +\infty}I(u_k(\lambda_kx))=\lim\limits_{k\rightarrow +\infty}\frac{1}{2}\int_{\mathbb{R}^4}|\Delta (u_k(\lambda_k x))|^2dx=\lim\limits_{k\rightarrow +\infty}\frac{1}{2}\int_{\mathbb{R}^4}|\Delta u_k|^2dx=M_{p}.
\end{equation}

\medskip

\emph{Step 2:} We claim that if $\gamma<\gamma^{*}$, then $M_p<\frac{16\pi^2}{\alpha_0}$.
\vskip0.1cm

We distinguish between the case $\gamma^{*}<+\infty$ and
$\gamma^{*}=+\infty$.
\vskip0.1cm

In the case $\gamma^{*}<+\infty$, since
$\gamma<\gamma^{*}$, then $\gamma<\gamma^{*}-\varepsilon_{0}$ for some
$\varepsilon_{0}>0$. It follows from the definition of  $\gamma^{*}$ that
 there exists some $u_{0}\in W^{2,2}(\mathbb{R}^{4})$ with $\Vert\Delta
u_{0}\Vert_{2}^{2}\leq R(F)$ satisfying
\[
\gamma^{*}-\varepsilon_{0}<\frac{2}{\Vert u_{0}\Vert_{2}^{2}}\int
_{\mathbb{R}^{4}}F(u_{0})dx.
\]
Consequently,
\[
\gamma\Vert u_{0}\Vert_{2}^{2}<2\int_{\mathbb{R}^{4}}F(u_{0})dx.
\]

Let $h(s)=\gamma\int_{\mathbb{R}^4}|su_0|^2dx-2\int_{\mathbb{R}^4}F(su_0)dx$ for $s>0$.
Since $h(1)<0$ and $h(s)>0$ for $s>0$ small enough, then there exists $s_{0}%
\in(0,1)$ satisfying $h(s_{0}u_{0})=0$. Therefore, we have  $s_{0}u_{0}\in
\mathcal{P}_\gamma$ and
\[
M_p\leq\frac{1}{2}\Vert\Delta(s_{0}u_{0})\Vert_{2}^{2}=\frac{1}{2}%
s_{0}^{2}\Vert\Delta u_{0}\Vert_{2}^{2}<\frac{1}{2}R(F)=\frac{16\pi^2}{\alpha_0}.
\]
\vskip0.1cm

In the case $\gamma^{*}=+\infty$, for any $\gamma>0$, there exists $u_{0}\in
W^{2,2}(\mathbb{R}^{4})$ with $\Vert\Delta u_{0}\Vert_{2}^{2}\leq R(F)$
satisfying
\[
\gamma\Vert u_{0}\Vert_{2}^{2}<2\int_{\mathbb{R}^{4}}F(u_{0})dx.
\]
Hence we can repeat the same arguments as  case $\gamma^{*}<+\infty$ to get the conclusion. Combining Step 1 and Step 2, we conclude that
If $\gamma<\gamma^{*}$, then $m_\gamma<\frac{16\pi^2}{\alpha_0}$.

\end{proof}

We now consider a minimizing sequence $\left\{  u_{k}\right\}  _{k}%
\subset\mathcal{N}_\gamma$ for $m_\gamma$. According to (A-R) condition (ii) and
$I(u_{k})\rightarrow m_\gamma>0$, we derive  that $\left\{
u_{k}\right\}  _{k}$ is bounded in $W^{2,2}(\mathbb{R}^4)$,
then up to a subsequence, there exists $u\in W^{2,2}(\mathbb{R}^4)$ such that

\begin{itemize}
\item $u_{k}\rightarrow u$ weakly in $W^{2,2}(\mathbb{R}^4)  $
and in $L^{p}\left(  \mathbb{R}^{4}\right)  $, for any $p>1$,

\item $u_{k}\rightarrow u$ in $L_{loc}^{p}\left(  \mathbb{R}^{4}\right)  $,

\item $u_{k}\rightarrow u$, a.e.
\end{itemize}
\begin{lemma}\label{lem3.3}
If $\gamma<\gamma^*$, then up to some translation, we can assume that the minimizing sequence $u_k$ satisfies $\lim\limits_{L\rightarrow +\infty}\lim\limits_{k\rightarrow +\infty}\int_{B_{L}}f(u_k)u_kdx\neq 0$.
\end{lemma}

\begin{proof}Define $M(L)=\lim\limits_{k\rightarrow +\infty}\sup_{y\in \mathbb{R}^{4}}\int_{B_{L}(y)}f(u_k)u_kdx$, we will show that $\lim\limits_{L\rightarrow +\infty}M(L)\neq 0$.  We first show that there exists some $R>0$ such that \begin{equation}\label{no-vanish}\lim\limits_{k\rightarrow \infty}\int_{\{|u_k|<R\}} f\left(u_{k}\right) u_{k} dx>0. \end{equation} Suppose not, that is
$$\lim\limits_{R\rightarrow \infty}\lim\limits_{k\rightarrow \infty}\int_{\{|u_k|<R\}} f\left(u_{k}\right) u_{k} \mathrm{~d}x=0,\ \
\lim\limits_{R\rightarrow \infty}\lim\limits_{k\rightarrow \infty}\int_{\{|u_k|\geq R\}} f\left(u_{k}\right) u_{k} \mathrm{~d}x=\lim\limits_{k\rightarrow \infty}\int_{\mathbb{R}^4} f\left(u_{k}\right) u_{k} \mathrm{~d}x.$$
Then it follows that $u_k\rightharpoonup 0$ in $W^{2,2}(\mathbb{R}^4)$ and $\{u_k\}_k$ blow-up. By the condition (ii), we know that  $\lim\limits_{t\rightarrow +\infty}\frac{f(t)t}{F(t)}=+\infty$. Hence, it follows that
\begin{equation}\begin{split}
\lim\limits_{k\rightarrow +\infty}\int_{\mathbb{R}^4}F(u_k)dx&=\lim\limits_{R\rightarrow +\infty}\lim\limits_{k\rightarrow +\infty}\int_{\{|u_k|<R\}}F(u_k)dx+\lim\limits_{R\rightarrow +\infty}\lim\limits_{k\rightarrow +\infty}\int_{\{|u_k|\geq R\}}F(u_k)dx\\
&=\lim\limits_{R\rightarrow +\infty}\lim\limits_{k\rightarrow +\infty}\int_{\{|u_k|\geq R\}}F(u_k)dx=0
\end{split}\end{equation}
From this  claim  and  Lemma \ref{m}, we  immediately get
 $\lim\limits_{k\rightarrow +\infty}\|\Delta u_k\|_{L^{2}(\mathbb{R}^4)}^2+\gamma \|u_k\|_{L^{2}(\mathbb{R}^4)}^2=2m_{\gamma}<\frac{32\pi^2}{\alpha_0}$. Through the sharp Adams inequality \eqref{Adams entire space} and Remark \ref{adrem}, we see that for any $\varepsilon>0$,

 \begin{equation}\begin{split}
0<\int_{\mathbb{R}^4} f\left(u_{k}\right)u_{k} dx&\leq \varepsilon\|u_k\|_2+C_{\varepsilon}\int_{\mathbb{R}^4}|u_k|^{\mu+1}(e^{\beta_0 u_k^2}-1)dx\\
&\leq \varepsilon\|u_k\|_2+c\|u_k\|^{\mu}_{\mu}\rightarrow 0 \text{ as } k\rightarrow +\infty,
\end{split}\end{equation}
 which is a contradiction. This proves that there exists some $R>0$ such that \eqref{no-vanish} holds.
\vskip0.1cm

Now, we are in position to prove that $\lim\limits_{L\rightarrow +\infty}M(L)\neq 0$. In fact, if $\lim\limits_{L\rightarrow +\infty}M(L)=0$, then for any $L>0$, $\lim\limits_{k\rightarrow +\infty}\sup_{y\in \mathbb{R}^{4}}\int_{B_{L}(y)}f(u_k)u_kdx=0$. It follows from the Lions lemma (\cite{Lions}) that $\lim\limits_{k\rightarrow +\infty}\|u\|_{L^{q}(\mathbb{R}^4)}=0$ for any $q>2$. Hence by Remark \ref{adrem}, we can derive that for any $R>0,$
\begin{equation}\begin{split}
\lim_{k\rightarrow \infty}\int_{\{|u_k|<R\}} |f\left(u_{k}\right) u_{k}|dx&\leq \lim\limits_{k\rightarrow \infty}\varepsilon\|u_k\|_2+C_{\varepsilon}\lim_{k\rightarrow \infty}\int_{\mathbb{R}^4}|u_k|^{\mu+1}(e^{\beta_0 R^2}-1)dx\\
&\leq C\varepsilon,
\end{split}\end{equation}
which is an contradiction with $\lim\limits_{k\rightarrow \infty}\int_{\{|u_k|<R\}} f\left(u_{k}\right) u_{k} \mathrm{~d}x>0$. This accomplishes the proof of $\lim\limits_{L\rightarrow +\infty}M(L)\neq 0$.
\vskip0.1cm

Hence there exists $x_k\in \mathbb{R}^{4}$ such that $\lim\limits_{L\rightarrow +\infty}\lim\limits_{k\rightarrow +\infty}\int_{B_{L}(x_k)}f(u_k)u_kdx\neq 0$. Denote $\{w_k\}_k$ by $w_k(x)=u_k(x+x_k)$, then $w_k$ is still a minimizing sequence for $m_\gamma$ and
$$\lim\limits_{L\rightarrow +\infty}\lim\limits_{k\rightarrow +\infty}\int_{B_{L}}f(w_k)w_kdx\neq 0.$$ For convenience, we still denoted this new minimizing sequence by $\{u_k\}_k$.

\end{proof}

Next, we claim that

\begin{lemma}\label{lem3.4}It holds
$$\lim_{k\rightarrow+\infty}\int_{\mathbb{R}^4}f(u_k)u_kdx=\lim_{L\rightarrow +\infty}\lim_{k\rightarrow+\infty}\int_{B_L(0)}f(u_k)u_kdx$$ and
$$\lim_{L\rightarrow +\infty}\lim_{k\rightarrow+\infty}\int_{\mathbb{R}^4\setminus B_L(0)}f(u_k)u_kdx=0.$$
\end{lemma}

\begin{proof}Set $$M=\lim\limits_{k\rightarrow+\infty}\int_{\mathbb{R}^4}f(u_k)u_kdx$$, $$M^0=\lim\limits_{L\rightarrow +\infty}\lim\limits_{k\rightarrow+\infty}\int_{B_L(0)}f(u_k)u_kdx$$ and $$M^{\infty}=\lim\limits_{L\rightarrow +\infty}\lim\limits_{k\rightarrow+\infty}\int_{\mathbb{R}^4\setminus B_L(0)}f(u_k)u_kdx.$$ Obviously, we have $M^{0}+M^{\infty}=M$.  We first show that
$(M^{0},M^{\infty})=(M,0)$ or $(M^{0},M^{\infty})=(0,M)$. Since $u_k\in \mathcal{N}_\gamma$, then
\begin{equation}\begin{split}
\int_{\mathbb{R}^4}(|\Delta u_k|^2+\gamma|u_k|^2)dx&=\lim\limits_{L\rightarrow +\infty}\lim\limits_{k\rightarrow+\infty}\int_{B_L(0)}f(u_k)u_kdx+\lim\limits_{L\rightarrow +\infty}\lim\limits_{k\rightarrow+\infty}\int_{\mathbb{R}^4\setminus B_L(0)}f(u_k)u_kdx\\
&=M^{0}+M^{\infty}.
\end{split}\end{equation}
Noticing that we can also write $\int_{\mathbb{R}^4}(|\Delta u_k|^2+\gamma|u_k|^2)dx$ as
\begin{align*}
 \int_{\mathbb{R}^4}(|\Delta u_k|^2+\gamma|u_k|^2)dx&=\lim\limits_{L\rightarrow +\infty}\lim\limits_{k\rightarrow+\infty}\int_{B_L(0)}(|\Delta u_k|^2+\\&+\gamma|u_k|^2)dx+\lim\limits_{L\rightarrow +\infty}\lim\limits_{k\rightarrow+\infty}\int_{\mathbb{R}^4\setminus B_L(0)}(|\Delta u_k|^2+\gamma|u_k|^2)dx.
\end{align*}

Hence we can assume that \begin{equation}\label{case1}\lim\limits_{L\rightarrow +\infty}\lim\limits_{k\rightarrow+\infty}\int_{B_L(0)}(|\Delta u_k|^2+\gamma|u_k|^2)dx\leq \lim\limits_{L\rightarrow +\infty}\lim\limits_{k\rightarrow+\infty}\int_{B_L(0)}f(u_k)u_kdx\end{equation} or  \begin{equation}\label{case2}\lim\limits_{L\rightarrow +\infty}\lim\limits_{k\rightarrow+\infty}\int_{\mathbb{R}^4\setminus B_L(0)}(|\Delta u_k|^2+\gamma|u_k|^2)dx\leq \lim\limits_{L\rightarrow +\infty}\lim\limits_{k\rightarrow+\infty}\int_{\mathbb{R}^4\setminus B_L(0)}f(u_k)u_kdx.\end{equation}
\vskip0.1cm

For a sufficiently large number $L>0$, define function
\begin{equation}\label{test}\phi_{L}^{0}(x)=\begin{cases}
1,\ \ \ \ \ {\rm if}\  |x|\leq L,\\
0\sim 1 ,\ {\rm if}\  L<|x|\leq L+1,\\
0,\ \ \ \ \ {\rm if}\  |x|\geq L+1,
\end{cases}\end{equation}
and $\phi_{L}^{\infty}(x)=1-\phi_{L}^{0}(x)$. Define $u_{k,L}^{*}=u_k\phi_{L}^{*}$ ($*=0$ or $\infty$), the following fact is easily proved by an  argument similar to that in \cite{Iko}.
$$\lim\limits_{L\rightarrow +\infty}\lim\limits_{k\rightarrow+\infty}\int_{\mathbb{R}^4}(|\Delta u_{k,L}^{0}|^2+\gamma|u_{k,L}^{0}|^2)dx=\lim\limits_{L\rightarrow +\infty}\lim\limits_{k\rightarrow+\infty}\int_{B_L(0)}(|\Delta u_k|^2+\gamma|u_k|^2)dx$$
$$\lim\limits_{L\rightarrow +\infty}\lim\limits_{k\rightarrow+\infty}\int_{\mathbb{R}^4}(|\Delta u_{k,L}^{\infty}|^2+\gamma|u_{k,L}^{\infty}|^2)dx=\lim\limits_{L\rightarrow +\infty}\lim\limits_{k\rightarrow+\infty}\int_{\mathbb{R}^4\setminus B_L(0)}(|\Delta u_k|^2+\gamma|u_k|^2)dx$$
$$\lim\limits_{L\rightarrow +\infty}\lim\limits_{k\rightarrow+\infty}\int_{\mathbb{R}^4}f(u_{k,L}^{0})u_{k,L}^0dx=\lim\limits_{L\rightarrow +\infty}\lim\limits_{k\rightarrow+\infty}\int_{B_L(0)}f(u_k)u_kdx$$
$$\lim\limits_{L\rightarrow +\infty}\lim\limits_{k\rightarrow+\infty}\int_{\mathbb{R}^4}f(u_{k,L}^{\infty})u_{k,L}^{\infty}dx=\lim\limits_{L\rightarrow +\infty}\lim\limits_{k\rightarrow+\infty}\int_{\mathbb{R}^4\setminus B_L(0)}f(u_k)u_kdx.$$
$$\lim\limits_{L\rightarrow +\infty}\lim\limits_{k\rightarrow+\infty}\int_{\mathbb{R}^4}F(u_{k,L}^{0})dx=\lim\limits_{L\rightarrow +\infty}\lim\limits_{k\rightarrow+\infty}\int_{B_L(0)}F(u_k)dx$$
$$\lim\limits_{L\rightarrow +\infty}\lim\limits_{k\rightarrow+\infty}\int_{\mathbb{R}^4}F(u_{k,L}^{\infty})dx=\lim\limits_{L\rightarrow +\infty}\lim\limits_{k\rightarrow+\infty}\int_{\mathbb{R}^4\setminus B_L(0)}F(u_k)dx.$$
Without loss of generality, we can assume that (\ref{case1}) holds, then there exists $t_{k,L}^{0}$ such that $t_{k,L}^{0}u_{k,L}^0\in \mathcal{N}_\gamma$. Obviously, $\lim\limits_{L\rightarrow +\infty}\lim\limits_{k\rightarrow+\infty}t_{k,L}^{0}\leq 1$. If
$t_{k,L}^{0}\leq 1$, then
\begin{equation}\label{eq3,1}\begin{split}
I_\gamma(t_{k,L}^{0}u_{k,L}^0)&=\frac{1}{2}\int_{\mathbb{R}^4}\big(f(t_{k,L}^{0}u_{k,L}^0)t_{k,L}^{0}u_{k,L}^0-2F(t_{k,L}^{0}u_{k,L}^0)\big)\\
&\leq \frac{1}{2}\int_{\mathbb{R}^4}\big(f(u_{k,L}^0)u_{k,L}^0-2F(u_{k,L}^0)\big).
\end{split}\end{equation}
If $t_{k,L}^{0}\geq 1$, then
\begin{equation}\label{eq3,2}\begin{split}
I_\gamma(t_{k,L}^{0}u_{k,L}^0)&=\frac{1}{2}(t_{k,L}^{0})^2\int_{\mathbb{R}^4}(|\Delta u_{k,L}^{0}|^2+\gamma|u_{k,L}^{0}|^2)dx-\int_{\mathbb{R}^4}F(t_{k,L}^{0}u_{k,L}^0)dx\\
&\leq \frac{1}{2}(t_{k,L}^{0})^2\int_{\mathbb{R}^4}(|\Delta u_{k,L}^{0}|^2+\gamma|u_{k,L}^{0}|^2)dx-\int_{\mathbb{R}^4}F(u_{k,L}^0)dx\\
&\leq \frac{1}{2}(t_{k,L}^{0})^2(\int_{\mathbb{R}^4}f(u_{k,L}^0)u_{k,L}^0dx+o_{k,L}(1))-\int_{\mathbb{R}^4}F(u_{k,L}^0)dx.
\end{split}\end{equation}
Combining the above estimate, we derive that
\begin{equation}\begin{split}
m_\gamma &\leq\lim\limits_{L\rightarrow +\infty}\lim\limits_{k\rightarrow+\infty}I_\gamma(t_{k,L}^{0}u_{k,L}^0)\\&\leq \frac{1}{2}\lim\limits_{L\rightarrow +\infty}\lim\limits_{k\rightarrow+\infty}\int_{\mathbb{R}^4}\big(f(u_{k,L}^0)u_{k,L}^0-2F(u_{k,L}^0)\big)dx\\
&\ \ \ \ \ +\frac{1}{2}\lim\limits_{L\rightarrow +\infty}\lim\limits_{k\rightarrow+\infty}\int_{\mathbb{R}^4}\big(f(u_{k,L}^{\infty})u_{k,L}^{\infty}-2F(u_{k,L}^{\infty})\big)dx\\
&=\lim\limits_{k\rightarrow +\infty}I(u_k)=m_\gamma.
\end{split}\end{equation}
Thus, we can conclude that $$\lim\limits_{L\rightarrow +\infty}\lim\limits_{k\rightarrow+\infty}\int_{\mathbb{R}^4}\big(f(u_{k,L}^{\infty})u_{k,L}^{\infty}-2F(u_{k,L}^{\infty})\big)dx=0,$$
that is $$\lim\limits_{L\rightarrow +\infty}\lim\limits_{k\rightarrow+\infty}\int_{\mathbb{R}^4\setminus B_{L}(0)}\big(f(u_{k})u_{k}-2F(u_{k})\big)dx=0,$$
which together with (A-R) condition implies that $M^{\infty}=0$, that is $(M^{0},M^{\infty})=(M,0)$. Similarly, we can prove that $(M^{0},M^{\infty})=(0,M)$ if we assume that
(\ref{case2}) holds.

Now, it remains to show that $(M^{0},M^{\infty})=(0,M)$ is impossible to occur. In fact, according to Lemma \ref{lem3.3}, we get $\lim\limits_{L\rightarrow +\infty}\lim\limits_{k\rightarrow +\infty}\int_{B_{L}(0)}f(u_k)u_kdx\neq 0$, which implies that $M^{0}\neq 0$. This accomplishes the proof of Lemma \ref{lem3.4}.
\end{proof}

\begin{lemma}\label{lem3.5be}
There holds $\lim\limits_{k\rightarrow +\infty}\int_{\mathbb{R}^4}F(u_k)dx=\int_{\mathbb{R}^4}F(u)dx$.
\end{lemma}
\begin{proof}
It follows from Lemma \ref{lem3.4} that $\lim\limits_{L\rightarrow +\infty}\lim\limits_{k\rightarrow +\infty}\int_{\mathbb{R}^{4}\setminus B_{L}(0)}f(u_k)u_kdx=0$, which together with the (A-R) condition implies that $\lim\limits_{L\rightarrow +\infty}\lim\limits_{k\rightarrow +\infty}\int_{\mathbb{R}^{4}\setminus B_{L}(0)}F(u_k)dx=0$. In order to obtain the desired convergence, we only need to prove that
$$\lim\limits_{L\rightarrow +\infty}\lim\limits_{k\rightarrow +\infty}\int_{B_{L}(0)}F(u_k)dx=\int_{\mathbb{R}^4}F(u)dx.$$
Indeed, for any $s>0$, we have
\begin{equation}\begin{split}
& |\int_{B_{L}(0)}F\left(u_k\right)  dx-\int_{B_{L}(0)}F\left(  u\right)  dx| \\
&  \leq |\int_{B_{L}(0)\cap \{|u_k|<s\}}
F\left(  u_k\right) dx-\int_{B_{L}(0)\cap \{|u_k|<s\}}F\left(  u\right)  dx| \\
&  \ \ \ \ +\left\vert \int_{B_{L}(0)\cap \{|u_k|\geq s\}}F\left(u_{k}\right)dx-\int_{B_{L}(0)\cap \{|u_k|\geq s\}}F\left(  u\right)  dx\right\vert \\
&  =I_{k,R,s}+II_{k,R,s}.
\end{split}\end{equation}
A direct application of the dominated convergence theorem leads to
$I_{k,R,s}\rightarrow0$. For $II_{k,R,s}$, from the condition (iii), we have
\begin{align*}
\int_{B_{L}^{0}\cap\left\{  {|u_{k}|}\geq s\right\}  }F\left(
u_{k}\right)  dx  &  \leq\frac{c}{s}\int_{\mathbb{R}^{4}\cap\left\{
{|u_{k}|}\geq s\right\}  }f\left(  u_{k}\right)  {u_{k}%
}dx\\
&  =\frac{c}{s}\int_{\mathbb{R}^{4}}f\left(  u_{k}\right)  {u_{k}%
}dx\rightarrow0,\text{ as }s\rightarrow\infty,
\end{align*}
where we have used the fact that $\int_{\mathbb{R}^{4}}f\left(  u_{k}\right)
{u_{k}}dx$ is bounded. Consequently, $II_{k,R,s}\rightarrow0$, and the lemma
is finished.
\vskip0.1cm
\end{proof}

\begin{lemma}\label{lem3.6}
Let $u_k$ be a bounded sequence in $W^{2,2}(\mathbb{R}^4)$ converging weakly to non-zero $u$. Furthermore, we also assume that $\lim\limits_{k\rightarrow+\infty}I_{\gamma}(u_k)<\frac{16\pi^2}{\alpha_0}$ and $\int_{\mathbb{R}^4}\big(|\Delta u|^2+\gamma|u|^2\big)dx>\int_{\mathbb{R}^4}f(u)udx$, then
$$\lim_{k\rightarrow+\infty}\int_{\mathbb{R}^4}f(u_k)u_kdx=\int_{\mathbb{R}^4}f(u)udx.$$
\end{lemma}

\begin{proof}
According to Lemma \ref{lem3.4}, we only need to prove that $$\lim_{L\rightarrow +\infty}\lim_{k\rightarrow+\infty}\int_{B_L(0)}f(u_k)u_kdx=
\int_{\mathbb{R}^4}f(u)udx.$$
It follows the lower semicontinuity of the norm in $W^{2,2}(\mathbb{R}^4)$ that
$$\lim\limits_{k\rightarrow\infty}\int_{\mathbb{R}^4}\big(|\Delta u_k|^2+\gamma|u_k|^2\big)dx\geq \int_{\mathbb{R}^4}\big(|\Delta u|^2+\gamma|u|^2\big)dx.$$
We divide the proof into the following case.
\vskip0.1cm

Case 1: $\int_{\mathbb{R}^4}\big(|\Delta u_k|^2+\gamma|u_k|^2\big)dx=\int_{\mathbb{R}^4}\big(|\Delta u|^2+\gamma|u|^2\big)dx$, then according to convexity of the norm and the equivalence of norms, we see that
$u_k\rightarrow u$ in $W^{2,2}(\mathbb{R}^4)$, hence $u_k\rightarrow u$ in $L^{p}(\mathbb{R}^4)$ for any $p\geq 2$. Hence
it follows from Adams inequality in $W^{2,2}(\mathbb{R}^4)$ that for any $p_0>1$, $\sup_{k}\int_{\mathbb{R}^4} \big(f(u_k)u_k\big)^{p_0}dx<\infty$,
which implies that
\begin{equation}\label{con1}
\lim_{L\rightarrow+\infty}\lim_{k\rightarrow\infty}\int_{B_{L}}f(u_k)u_kdx=\lim_{L\rightarrow+\infty}\int_{B_{L}}f(u)udx=\int_{\mathbb{R}^4}f(u)udx.
\end{equation}

Case 2: If $\lim\limits_{k\rightarrow\infty}\int_{\mathbb{R}^4}\big(|\Delta u_k|^2+\gamma|u_k|^2\big)dx>\int_{\mathbb{R}^4}\big(|\Delta u|^2+\gamma|u|^2\big)dx$, we set
 $$ v_k:=\frac{u_k}{\lim\limits_{k\rightarrow\infty}((\|\Delta u_k\|_2^2+\gamma\|u_k\|_2^2))^{\frac{1}{2}}}\ \mbox{and}\ v_0:=\frac{u}{\lim\limits_{k\rightarrow\infty}(\|\Delta u_k\|_2^2+\gamma\|u_k\|_2^2)^{\frac{1}{2}}}.$$
 We claim there exists $q_0>1$ sufficiently $1$ such that
 \begin{equation}\label{d.7}
  q_0(\|\Delta u_k\|_2^2+\gamma\|u_k\|_2^2)<\frac{32\pi^2}{1-(\|\Delta v_0\|_2^2+\gamma\|v_0\|_2^2)}.
  \end{equation}
 Indeed, we can apply the Lemma \ref{lem3.5be} and (A-R) condition to obtain
  \begin{equation}\begin{split}\label{d.8}
  &\lim\limits_{k\rightarrow\infty}((\|\Delta u_k\|_2^2+\gamma\|u_k\|_2^2))\big(1-(\|\Delta v_0\|_2^2+\gamma\|v_0\|_2^2)\big)\\
  &\ \ =\lim\limits_{k\rightarrow\infty}(\|\Delta u_k\|_2^2+\gamma\|u_k\|_2^2))\Big(1-\frac{\|\Delta u\|_2^2+\gamma\|u\|_2^2)}{\|\Delta u_k\|_2^2+\gamma\|u_k\|_2^2)}\Big)\\
 &\ \ =2\lim\limits_{k\rightarrow +\infty}I_{\gamma}(u_k)+2\int_{\mathbb{R}^4}F(u_k)dx-2I(u)-2\int_{\mathbb{R}^4}F(u)dx\\
&\ \ <\frac{32\pi^2}{\alpha_0},
\end{split}\end{equation}
where the last inequality holds because $\int_{\mathbb{R}^4}\big(|\Delta u|^2+\gamma|u|^2\big)dx>\int_{\mathbb{R}^4}f(u)udx>2\int_{\mathbb{R}^4}F(u)udx$.
Combining the above estimate with Adams inequality in $W^{2,2}(\mathbb{R}^4)$, one can derive that there exists $p_0>1$ such that
\begin{eqnarray}\label{d.9}
\sup_{k}\int_{\mathbb{R}^4}\big(f(u_k)u_k\big)^{p_0}dx<\infty.
\end{eqnarray}
Then it follows Vitali convergence theorem that $$\lim_{L\rightarrow+\infty}\lim_{k\rightarrow\infty}\int_{B_{L}}f(u_k)u_kdx=\lim_{L\rightarrow+\infty}\int_{B_{L}}f(u)udx=\int_{\mathbb{R}^4}f(u)udx.$$
then we accomplishes the proof of Lemma \ref{lem3.6}.
\end{proof}

Now we are in position to give the existence of ground-state solutions for the bi-harmonic equation with the constant potential $\gamma<\gamma^*$.
\begin{proof}[Proof of the first part of Theorem \ref{thm1}]
Since $\gamma<\gamma^{*}$, we  will prove that $m_{\gamma}$ is achieved by some non-zero function $u$. We argue this by contradiction. Suppose that some $u=0$, then
\begin{equation}\begin{split}
\lim_{k\rightarrow \infty}(\|\Delta u_k\|_2^2+\gamma\|u_k\|_2^2)&=2\lim_{k\rightarrow \infty}I_{\lambda}(u_k)+2\int_{\mathbb{R}^4}F(u_k)dx\\
&=2\lim_{k\rightarrow \infty}I_{\lambda}(u_k)=2m_\gamma<\frac{32\pi^2}{\alpha_0}
\end{split}\end{equation}
Then it follows from the Adams inequality in $\mathbb{R}^4$ and Lemma \ref{lem3.4} that
$$\lim_{k\rightarrow \infty}\int_{\mathbb{R}^4}f(u_k)u_kdx=0,$$
which implies that $$0<m_\gamma=\lim_{k\rightarrow \infty}(\|\Delta u_k\|_2^2+\gamma\|u_k\|_2^2)=\lim_{k\rightarrow \infty}\int_{\mathbb{R}^4}f(u_k)u_kdx=0,$$
which is a contradiction. This proves $u\neq 0$.

Next, we claim that $$\|\Delta u\|_2^2+\gamma\|u\|_2^2\leq \int_{\mathbb{R}^4}f(u)udx.$$
Suppose this is false, that is,
\begin{equation}\label{t1}
\|\Delta u\|_2^2+\gamma\|u\|_2^2>\int_{\mathbb{R}^4}f(u)udx.
\end{equation}
In view of Lemma \ref{m} and Lemma \ref{lem3.6}, we derive that  $$\lim_{k\rightarrow \infty}\int_{\mathbb{R}^4}f(u_k)u_kdx=\int_{\mathbb{R}^4}f(u)udx.$$
This implies that
\begin{equation}\begin{split}
\|\Delta u_0\|_2^2+\gamma\|u_0\|_2^2&\leq \|\Delta u_k\|_2^2+\gamma\|u_k\|_2^2\\
&=\lim_{k\rightarrow \infty}\int_{\mathbb{R}^4}f(u_k)u_kdx\\
&=\int_{\mathbb{R}^4}f(u)udx<\|\Delta u\|_2^2+\gamma\|u\|_2^2,
\end{split}\end{equation}
which is a contradiction. This proves the claim.
\vskip 0.1cm
Since $$\|\Delta u\|_2^2+\gamma\|u\|_2^2\leq \int_{\mathbb{R}^4}f(u)udx,$$
there exists $\gamma_{0}\in (0,1]$ such that $\gamma_{0}u\in \mathcal{N}_\gamma$. According to the definition of  $m_{\gamma}$, we derive that
\begin{equation}\begin{split}
m_{\gamma}\leq I_\gamma(\gamma_{0}u)&=\frac{1}{2}\int_{\mathbb{R}^4}\big(f(\gamma_{0}u)(\gamma_{0}u)-2F(\gamma_{0}u)\big)dx\\
&\leq \frac{1}{2}\int_{\mathbb{R}^4}\big(f(u)(u)-2F(u)\big)dx\\
&\leq \lim_{k\rightarrow \infty}\frac{1}{2}\int_{\mathbb{R}^4}\big(f(u_k)(u_k)-2F(u_k)\big)dx\\
&=\lim_{k\rightarrow \infty}I_{\gamma}(u_k)=m_{\gamma}.
\end{split}\end{equation}
This implies that $\gamma_0=1$ and $u\in \mathcal{N}_\gamma$ and $I_\gamma(u)=m_{\gamma}$. This means that the equation
(\ref{3.1}) admits a ground-state solution for any $\gamma\in(0,\gamma^{\ast})$.
\medskip
\end{proof}

\bigskip
In order to finish the proof of Theorem \ref{thm1}, we need the following result.

\begin{lemma}\label{lem3.8}
 $m_\gamma<\frac{16\pi^2}{\alpha_0}$ if and only if $\gamma<\gamma^{*}$.
\end{lemma}
\begin{proof}
Recalling Lemma \ref{m}, we have proved that if $\gamma<\gamma^{*}$, then $m_\gamma<\frac{16\pi^2}{\alpha_0}$. Hence we only need to prove that if $m_\gamma<\frac{16\pi^2}{\alpha_0}$, then $\gamma<\gamma^{*}$.
Obviously, if the $\gamma^{*}=+\infty$, then $\gamma<\gamma^{*}$ and the
proof is complete. Therefore, without loss of generality, we may assume that
$\gamma^{*}<+\infty$. From the previously discussion, we know that if $m_\gamma<\frac{16\pi^2}{\alpha_0}$, then $m_{\gamma}$ could be achieved by some function $u\in W^{2,2}(\mathbb{R}^4)\setminus \{0\}$ which is a ground state solution to equation \eqref{3.1}. Obviously, we have $u\in \mathcal{P}_{\gamma}$, which implies that $M_{p}\leq m_{\gamma}$. Recalling Lemma \ref{m}, we have already proved that $M_p\geq m_\gamma$. Combining these facts, we conclude that $M_p$ is also achieved by $u_\gamma$.
Then according to the definition of the $M_{p}$, we have $\left\Vert
\Delta u\right\Vert _{2}^{2}<32\pi^{2}/\alpha_{0}$ and $\gamma\left\Vert
u\right\Vert _{2}^{2}=2\int_{\mathbb{R}^{4}}F(u)dx$. Define
\[
g(s)=\frac{2}{s^{2}\left\Vert u\right\Vert _{2}^{2}}\int_{\mathbb{R}^{4}%
}F(su)dx,
\]
then $g(1)=\gamma$. From the (A-R) condition, then it is easy to
see that $g(s)$ is monotone increasing. If we set $v=\frac{R(F)^{1/2}%
}{\left\Vert \Delta u\right\Vert _{2}}u$, then $\Vert\Delta v\Vert_{2}%
^{2}=R(F)$ and \
\[
\gamma^{*}\geq\frac{2}{\Vert v\Vert_{2}^{2}}\int_{\mathbb{R}^{4}%
}F(v)dx=g(\frac{R(F)^{1/2}}{\Vert\Delta u\Vert_{2}})>g(1)=\gamma.
\]
\end{proof}
Now, we give the proof for the non-existence of ground-state solutions for the bi-harmonic equation with the constant potential $\gamma>\gamma^*$.
\begin{proof}[Proof of the second part of Theorem \ref{thm1}]
 We argue this by contradiction. We assume that there exists $\gamma_0>\gamma^{*}$ such that the equation (\ref{3.1}) admits a ground-state solution. From Lemma \ref{upper} and Lemma \ref{lem3.8}, we know that $m_{\gamma_0}=\frac{16\pi^2}{\alpha_0}$. Since $m_{\gamma_0}$ could be achieved by some function $u_0\in W^{2,2}(\mathbb{R}^4)$, Direct calculation gives that for any $\gamma \in (\gamma^{*},\gamma_0)$, $m_{\gamma}<m_{\gamma_0}=\frac{16\pi^2}{\alpha_0}$. This is contradiction with the fact: $m_\gamma<\frac{16\pi^2}{\alpha_0}$ if and only if $\gamma<\gamma^{*}$.  This indicates that for any $\gamma>\gamma^{*}$, equation
(\ref{3.1}) does not admit a ground-state solution.

\end{proof}

\section{The Adams inequality with degenerate potentials in $\mathbb{R}^4$: Proof of Theorem \ref{Ad}}
In this section, we will prove the critical Adams inequality involving degenerate potential, namely we will give the proof of Theorem \ref{Ad}. For this purpose, we need the following lemma.
\begin{lemma}\label{imbedding}
Assume that $u\in W^{2,2}(\mathbb{R}^4)$ such that $\int_{\mathbb{R}^4}\left(|\Delta u|^2+V(x)u^2\right)dx<+\infty$, where
$V(x)$ satisfies the assumption (V). Then there exits some constant $c>0$ depending on $\delta$ and $c_0$ such that $$\int_{\mathbb{R}^4}u^2dx\le c\int_{\mathbb{R}^4}\left(|\Delta u|^2+V(x)u^2\right)dx.$$.
\end{lemma}

\begin{proof}
Choose the cutoff function $\eta$ such that $\eta=1$ in $B_{2\delta} $ and $\eta=0$ in  $\mathbb{R}^4\backslash B_{4\delta }$. Obviously, $|\eta|\leq 1$ and  $|\Delta \eta|\leq \frac{c}{\delta^2}$. By the Poincare inequality and Young inequality, we derive that

\begin{align*}
   \int\limits_{{B_{4\delta }}} {{{\left| {u\eta } \right|}^2}} dx & \le c{\delta ^4}\int\limits_{{B_{4\delta }}} {{{\left| {\Delta\left( {u\eta } \right)} \right|}^2}} dx \\
  & \le  c{\delta ^4}\int\limits_{{B_{4\delta }}} {{{\left| {\eta \Delta u + u\Delta \eta +\nabla u \nabla \eta} \right|}^2}} dx \\
  &  \le c{\delta ^4}\int\limits_{{B_{4\delta }}} {{{\left| {\eta \Delta u} \right|}^2}} dx +c\delta^4\int\limits_{{B_{4\delta }}\backslash {B_{2\delta} }} {{{\left| u \right|}^2|\Delta \eta|^2}} dx + c\delta^4 \int_{B_{4\delta}\setminus B_{2\delta}}|\nabla u \nabla \eta|^2dx  \\
  &\leq c{\delta ^4}\int\limits_{{B_{4\delta }}} {{{\left| {\Delta u} \right|}^2}}dx+c\int\limits_{{B_{4\delta }}\backslash {B_{2\delta} }} {{{\left| u \right|}^2}} dx+c\delta^2 \int_{B_{4\delta}\setminus B_{2\delta}}\big(|u|^2+|\Delta u|^2\big)dx.
  \end{align*}

This gives that there exists $c_1$ depending on $\delta$ such that $$\int_{B_{2\delta}}|u_k|^2dx\leq c_1\int_{B_{4\delta}\setminus B_{2\delta}}|u|^2dx+c_1\int_{B_{4\delta}}|\Delta u|^2dx,$$
which together with $V(x)\geq c_0$ in $\mathbb{R}^4\setminus B_{2\delta}$ implies that $$\int_{\mathbb{R}^4}u^2dx< c\int_{\mathbb{R}^4}\left(|\Delta u|^2+V(x)u^2\right)dx,$$
where $c$ depends on $c_0$ and $\delta$.

\end{proof}

Now, we are in position to prove Theorem \ref{Ad}.
\medskip

\emph{Proof of Theorem \ref{Ad}:}
Since $C_{c}^{\infty}(\mathbb{\mathbb{R}}^4)$ is dense in $W^{2,2}(\mathbb{R}^4)$, we may assume that $u$ is a compactly supported smooth function.
Furthermore, we assume that $\int_{\mathbb{R}^4}V(x)u^2dx>0$. In fact, if $\int_{\mathbb{R}^4}V(x)u^2dx=0$, then obviously $\textrm{supp } u\subseteq B_{2\delta}(0)$, through the classical Adams inequality on bounded domain, we have
$$\int_{\mathbb{R}^2}\big(e^{32\pi^2 u^2}-1\big)dx=\int_{B_{2\delta}}\big(e^{32\pi^2 u^2}-1\big)dx<c,$$  and the proof of Theorem \ref{Ad} is completed.

Hence it remains to consider the case when $\int_{\mathbb{R}^4}V(x)u^2dx>0$. Set $$A(u):=\left(\int_{\mathbb{R}^4}V(x)u^2dx\right)^{\frac{1}{2}}$$
and  $ \Omega (u):=\left \{ x\in \mathbb{R}^4| \ u>A(u) \right \}$ and $ \tilde{\Omega} (u):=\left \{ x\in \mathbb{R}^4| \ -u>A(u) \right \}$. Then $A(u)<1$ and
\begin{align*}
\int_{\Omega (u)\cap B^c_{2\delta}}u^2dx& \geq  \int_{\Omega (u)\cap B^c_{2\delta}}A^2(u)dx\\
 & = \left(\int_{\mathbb{R}^4}V(x)u^2dx\right) \left |   \Omega (u)\cap B^c_{2\delta}     \right |,
\end{align*}
hence by (V1) we get $$ \left |   \Omega (u)\cap B^c_{2\delta} \right |\leqslant \frac{ \int_{\Omega (u)\cap B^c_{2\delta}}u^2dx}{\int_{\mathbb{R}^4}V(x)u^2dx} \leq \frac{1}{c_0},$$
and then $\left | \Omega (u)\right |\leq |B_{2\delta}|+\frac{1}{c_0}.$ Similar, we can also obtain $\left | \tilde{\Omega} (u)\right |\leq |B_{2\delta}|+\frac{1}{c_0}.$

Now, we rewrite \begin{align*}\int_{\mathbb{R}^{4}}\left (  e^{32\pi^2 |u|^{2}}-1 \right )dx&= \int_{\Omega (u)}\left (  e^{32\pi^2 |u|^{2}}-1 \right )dx+ \int_{\tilde{\Omega} (u)}\left (  e^{32\pi^2 |u|^{2}}-1 \right )dx \\
&\ \ \ +\int_{\mathbb{R}^4\setminus \ (\Omega (u)\cup \tilde{\Omega} (u))}\left ( e^{32\pi^2 |u|^{2}}-1 \right )dx\\
&:=I_{1}+I_{2}+I_{3},\end{align*}
and we will prove that both $I_1$, $I_2$ and $I_3$ are bounded by a constant $c$.
\vskip0.1cm

First, we estimate $I_3$. Since $\int_{\mathbb{R}^4}\left(|\Delta u|^2+V(x)u^2\right)dx\leq 1$, through Lemma \ref{imbedding}, we know that $\int_{\mathbb{R}^4}u^2dx$ is also bounded by some constant, then

\begin{align*} I_{3}  \leqslant \int_{\{|u(x)|<1\}}  \sum_{k=1}^{\infty} \frac{\left ( 32\pi^2 \right ) ^{k}}{k !}|u|^{2k}dx  \leqslant \sum_{k=1}^{\infty} \frac{\left ( 32\pi^2 \right ) ^{k}}{k !} \int_{\mathbb{R}^4}u^2dx  \leqslant c. \end{align*}

Since the estimate of $I_1$ and $I_3$ is similar, we only estimate $I_1$. Set $$v(x)=u(x)-A(u) \quad \text{in } \Omega(u),$$ then $v \in W^{2,2}(\Omega(u))$ with $v=0$ on the $\partial\Omega(u)$. Direct calculation gives that in $\Omega(u)$,
\begin{align*}
u^2(x) &=\left ( v(x)+A(u) \right ) ^2 \\
&=v^2(x)+A^2(u)+2v(x)A(u)\\
&\le v^2(x)+A^2(u)+v^2(x)A^2(u)+1\\
&=v^2(x)(1+A^2(u))+A^2(u)+1.
\end{align*}

Let $w(x)=v(x)(1+A^2(u)) ^{\frac{1}{2} }$, then
$w(x)\in W^{2,2}( \Omega(u))\cap W^{1,2}_{0}(\Omega(u)) $, $$u^2(x)\le w^2(x)+1+A^2(u),$$  $$\Delta w(x)=(1+A^2(u))^{\frac{1}{2} }\Delta v(x),$$
and \begin{align*}
\int_{\Omega\left ( u \right )  } \left | \Delta w(x) \right |^2dx &= \left ( 1+A^2(u) \right )\int_{\Omega\left ( u \right ) }\left | \Delta v(x) \right |^2dx\\
&\le  \left ( 1+A^2(u)\right )\left ( 1-\int_{\mathbb{R}^4}V(x)u^2dx  \right ) \\
&=\left ( (1+\int_{\mathbb{R}^4}V(x)u^2dx ) \right )\left ( 1-\int_{\mathbb{R}^4}V(x)u^2dx  \right ) \\
&\le 1.
\end{align*}
Then using the Adams inequalities on bounded domain with the Navier boundary  (see \cite{Ta}), we get

\begin{align*}
I_1&\le \int_{\Omega (u)}\left (  e^{32\pi^2 |u|^{2}}-1 \right )dx
\le  e^{32\pi^2(1+A^2(u))}\int_{\Omega (u)} e^{32\pi^2 |w|^{2}} dx\le c,
 \end{align*}
and the proof of Theorem \ref{Ad} is finished.

\section{Existence of the ground-state solution of bi-harmonic equations with degenerate potential: The proof of Theorem \ref{thm3}}

In this section, we are concerned with the ground states of the following quasilinear bi-harmonic equation (\ref{degen}), where $f(t)$ has the critical exponential growth satisfying (i)-(v) and the potential $V\left(  x\right)\geq 0$ satisfies (V1) and (V2).

The associated functional and Nehari Manifold are
$$I_{V}(u)=\frac{1}{2}\int_{\mathbb{R}^{4}}\left(  \left\vert \Delta
u\right\vert ^{2}+V(x)\left\vert u\right\vert ^{2}\right)
dx-\int_{\mathbb{R}^{4}}F(u)dx$$ and
$$\mathcal{N}_{V}=\left\{  \left.  u\in W^{2,2}\left(  \mathbb{R}^{4}\right)
\right\vert u\neq0,N_V\left(  u\right)  =0\right\},$$
respectively, where
$$N_{V}\left(  u\right)  =\int_{\mathbb{R}^{4}}\left(  \left\vert \Delta
u\right\vert ^{2}+V(x)\left\vert u\right\vert ^{2}\right)
dx-\int_{\mathbb{R}^{4}}f(u)udx.$$
One can easily verify that if $u\in\mathcal{N}_V$, then
$$I_V(u)=\frac{1}{2}\int_{\mathbb{R}^4}(f(u)u-2F(u))dx.$$

Set $m_{V}=\inf \{I_V(u)\ |\ N_V(u)=0\}$, we will prove that if $V_{\infty}<\gamma^{*}$, then $m_{V}$ is achieved by some function $u\in W^{2,2}(\mathbb{R}^4)$.

\begin{lemma}
\label{impor}If $V_\infty<\gamma^{*}$, then
\begin{equation}
0<m_V<\frac{16\pi^2}{\alpha_0}. \label{the rang}%
\end{equation}
\end{lemma}

\begin{proof}
We first show that $m_V>0$.  We prove this by contradiction. Assume that there
exists some sequence $u_{k}\in$ $\mathcal{N}_V$ such that $I_V\left(
u_{k}\right)  \rightarrow0$, that is,%
\[
\lim\limits_{k\rightarrow+\infty}\int_{\mathbb{R}^{4}}\big(f(u_{k}%
)u_{k}-2F(u_{k})\big)dx=0,
\]
which together with (A-R) condition and $u_{k}\in\mathcal{N}_V$ yields that
\begin{equation}
\lim_{k\rightarrow+\infty}\int_{\mathbb{R}^{4}}\left(|\Delta u_{k}|^{2}%
+V(x)|u_{k}|^{2}\right)dx=0.\label{t1c1}%
\end{equation}
On one hand, it follows from \eqref{xin1} and $u_{k}\in\mathcal{N}_V$ that
\begin{equation}%
\begin{split}
1&=\int_{\mathbb{R}^{4}}f(u_{k})\frac{u_{k}}{\Vert u_{k}\Vert^2_{W^{2,2}_{V}(\mathbb{R}^4)}}dx\\
&\leq\int_{\mathbb{R}^{4}}\big(\varepsilon\frac{|u_{k}|^{2}}{\Vert
u_{k}\Vert^2_{W^{2,2}_V(\mathbb{R}^4)}}+C_{\varepsilon}\frac{|u_{k}|^{\mu+1}}{\Vert u_{k}\Vert_{W^{2,2}_V(\mathbb{R}^4)}%
^{2}}\left(  e^{\beta_{0}u_{k}^{2}}-1\right)\big)dx
\end{split}
\label{t2}%
\end{equation}

On the other hand, by \eqref{t1c1},  Adams inequality involving the degenerate potential (Theorem \ref{Ad}) and the fact
that $\mu>2$, we get for any $p>1,$
\begin{align*}
\frac{1}{\Vert u_{k}\Vert^2_{W^{2,2}_{V}(\mathbb{R}^4)}}\int_{\mathbb{R}^{4}}%
|u_{k}|^{\mu+1}e^{\beta_{0}u_{k}^{2}}dx  & \leq\frac{1}{\Vert u_{k}%
\Vert^2_{W^{2,2}_{V}(\mathbb{R}^4)}}\left(  \int_{\mathbb{R}^{4}}|u_{k}|^{(\mu+1) p}dx\right)
^{1/p}\left(  \int_{\mathbb{R}^{4}}\left(  e^{p^{\prime}\beta_{0}u_{k}^{2}%
}-1\right)  dx\right)  ^{1/p^{\prime}}\\
& \leq c\Vert u_{k}\Vert_{W^{2,2}_{V}(\mathbb{R}^4)}^{\mu-1}\rightarrow0,\text{as }k\rightarrow
\infty.
\end{align*}
which is a contradiction with \eqref{t2}.
\medskip

Next, we prove that $m_{V}<\frac{16\pi^2}{\alpha_0}$. Define \[
I_{\infty}\left(  u\right)  =\frac{1}{2}\int_{\mathbb{R}^{4}}\left(
\left\vert \Delta u\right\vert ^{2}+V_{\infty}\left\vert u\right\vert ^{2}\right)
dx-\int_{\mathbb{R}^{4}}F(u)dx
\]
and%
\[
\mathcal{N}_{\infty}=\left\{  \left.  u\in W^{2,2}\left(  \mathbb{R}^{4}\right)
\right\vert u\neq0,N_{\infty}\left(  u\right)  =0\right\}  ,
\]
where
\[
N_{\infty}\left(  u\right)  =\int_{\mathbb{R}^{4}}\left(  \left\vert \Delta
u\right\vert ^{2}+V_{\infty}\left\vert u\right\vert ^{2}\right)  dx-
\int_{\mathbb{R}^{4}}f(u)udx.
\]
Set \[
m_{\infty}=\inf\left\{  I_{\infty}\left(  u\right)  ,u\in\mathcal{N}%
_{\infty}\right\}\]
Recalling the proof of Theorem \ref{thm1}, we have proved that $m_\infty$ is achieved by some function $w\in W^{2,2}(\mathbb{R}^4)$ if $V_\infty<\gamma^{*}$ with
$m_\infty=I_{\infty}(w)<\frac{16\pi^2}{\alpha_0}$. Since $w\in \mathcal{N}_{\infty}$, then $\int_{\mathbb{R}^4}(|\Delta w|^2+V_{\infty}|w|^2)dx=\int_{\mathbb{R}^4}f(w)wdx$, which implies that $\int_{\mathbb{R}^4}(|\Delta w|^2+V(x)|w|^2)dx<\int_{\mathbb{R}^4}f(w)wdx$ from the assumption (V2). It follows that there exists $t\in (0,1)$ such that
$tw\in \mathcal{N}_{V}$ and
\begin{equation}\begin{split}
m_{V}\leq I_{V}(tw)&=\frac{1}{2}\int_{\mathbb{R}^4}\big(f(tw)tw-2F(tw)\big)dx\\
&< \frac{1}{2}\int_{\mathbb{R}^4}\big(f(w)w-2F(w)\big)dx\\&=I_{\infty}(w)=m_{\infty},
\end{split}\end{equation}
this proves that \begin{equation}\label{compare}m_{V}< m_{\infty}<\frac{16\pi^2}{\alpha_0}, \text{  if } V_\infty<\gamma^*,\end{equation}
and the proof for this lemma is finished.
\end{proof}

We now consider a minimizing sequence $\left\{  u_{k}\right\}  _{k}%
\subset\mathcal{N}_V$ for $m_V$. According to (A-R) condition (iii) and
$I_V(u_{k})\rightarrow m_V>0$, we derive  that $\left\{
u_{k}\right\}  _{k}$ is bounded in $W^{2,2}(\mathbb{R}^4)$,
then up to a subsequence, there exists $u\in W^{2,2}(\mathbb{R}^4)$ such that

\begin{itemize}
\item $u_{k}\rightarrow u$ weakly in $W^{2,2}(\mathbb{R}^4)  $
and in $L^{p}\left(  \mathbb{R}^{4}\right)  $, for any $p>1$,

\item $u_{k}\rightarrow u$ in $L_{loc}^{p}\left(  \mathbb{R}^{4}\right)  $,

\item $u_{k}\rightarrow u$, a.e.
\end{itemize}
We claim
\begin{lemma}\label{lem4.3}
$u\neq 0$.
\end{lemma}
\begin{proof}
We prove this by contradiction. If $u=0$, and $u_{k}\rightarrow0$
in $L_{loc}^{2}\left(  \mathbb{R}^{4}\right)  $. We first claim that:%
\begin{equation}
\underset{k\rightarrow+\infty}{\lim}\int_{\mathbb{R}^{4}}\left(
V_{\infty}-V\left(  x\right)  \right)  \left\vert u_{k}\right\vert ^{2}dx=0.
\label{vanish}%
\end{equation}

For any fixed $\varepsilon>0$, we take $R_{\varepsilon}>0$ such that
\[
\left\vert V_{\infty}-V\left(  x\right)  \right\vert \leq\varepsilon,\text{ for
any }\left\vert x\right\vert >R_{\varepsilon}.
\]
Combining this and the boundedness of $u_{k}$ in $W^{2,2}\left(  \mathbb{R}%
^{4}\right)  $, we derive that
\begin{align*}
\int_{\mathbb{R}^{4}}\left(V_{\infty}-V\left(  x\right)  \right)  \left\vert
u_{k}\right\vert ^{2}dx &  =\int_{B_{R_{\varepsilon}}}\left( V_{\infty}-V\left(
x\right)  \right)  \left\vert u_{k}\right\vert ^{2}dx+\int_{B_{R_{\varepsilon
}}^{c}}\left( v_{\infty}-V\left(  x\right)  \right)  \left\vert u_{k}\right\vert
^{2}dx\\
&  \leq c\int_{B_{R_{\varepsilon}}}\left\vert u_{k}\right\vert ^{2}%
dx+M\varepsilon,
\end{align*}
where $M=\underset{k}{\sup}\int_{\mathbb{R}^{4}}\left\vert u_{k}\right\vert
^{2}dx$. This together with $u_{k}\rightarrow0$ in $L_{loc}^{2}\left(
\mathbb{R}^{4}\right)  $ as $k\rightarrow\infty$ yields that
\[
\underset{k\rightarrow+\infty}{\lim}\int_{\mathbb{R}^{4}}\left(  \gamma-V\left(  x\right)  \right)  \left\vert
u_{k}\right\vert ^{2}dx\leq M\varepsilon,
\]
which implies (\ref{vanish}) holds.

Since $u_{k}\in \mathcal{N}_{V}$, we know that there exists some sequence
$t_{k}\geq1$ such that $t_{k}u_{k}\in\mathcal{N}_{\infty}$, that is,%
\begin{equation}
\int_{\mathbb{R}^{4}}\left(  \left\vert \nabla u_{k}\right\vert ^{2}%
+V_{\infty}\left\vert u_{k}\right\vert ^{2}\right)  dx-\int_{\mathbb{R}^{4}}%
\frac{f\left(  t_{k}u_{k}\right)  u_{k}}{t_{k}}dx=0.\label{sub1}%
\end{equation}
On the other hand, since $u_{k}\in \mathcal{N}_{V}$, then
\begin{equation}
\int_{\mathbb{R}^{4}}\left(  \left\vert \nabla u_{k}\right\vert ^{2}+V\left(
x\right)  \left\vert u_{k}\right\vert ^{2}\right)  dx-\int_{\mathbb{R}^{4}%
}f\left(  u_{k}\right)  u_{k}dx=0.\label{sub2}%
\end{equation}

Combining (\ref{sub1}) and (\ref{sub2}),\ we get%
\begin{align*}
&  \int_{\mathbb{R}^{4}}\frac{f\left(  t_{k}u_{k}\right)  }{t_{k}u_{k}}%
u_{k}^{2}dx-\int_{\mathbb{R}^{4}}\frac{f\left(  u_{k}\right)  }{u_{k}}%
u_{k}^{2}dx\\
&  =\int_{\mathbb{R}^{4}}\left(  V_{\infty}-V\left(  x\right)  \right)  \left\vert
u_{k}\right\vert ^{2}dx\rightarrow0.
\end{align*}
\ Hence
\begin{equation}
\int_{\mathbb{R}^{4}}\frac{f\left(  t_{k}u_{k}\right)  }{t_{k}u_{k}}u_{k}%
^{2}dx=\int_{\mathbb{R}^{4}}\frac{f\left(  u_{k}\right)  }{u_{k}}u_{k}%
^{2}dx+o_{k}\left(  1\right).  \label{converg}%
\end{equation}

Next, we claim that $t_{k}\rightarrow t_{0}=1$ as $k\rightarrow\infty$. We
prove this by contradiction. Assume that $t_{0}>1$. Since
\[
I_{V}\left(  u_{k}\right)  =\frac{1}{2}\int_{\mathbb{R}^{4}}\left(  f\left(
u_{k}\right)  u_{k}-2F\left(  u_{k}\right)  \right)  dx\rightarrow
m_{V}>0,\text{as }k\rightarrow\infty,
\] we have $\underset{k\rightarrow\infty}{\lim}\int_{\mathbb{R}%
^{4}}f\left(  u_{k}\right)  u_{k}dx>0$. Now, we claim that \begin{equation}\underset{k\rightarrow\infty} {\lim}  \int_{\mathbb{R}^{4}}%
\frac{f\left(  t_{0}u_{k}\right)  }{t_{0}u_{k}}u_{k}^{2}dx >\underset{k\rightarrow\infty}{\lim}\int_{\mathbb{R}%
^{4}}f\left(  u_{k}\right)  u_{k}dx.\end{equation}
Since $f(t)=o(t)$ as $t\rightarrow0$, and $\|u_k\|_2$ is bounded, then $$\lim\limits_{r\rightarrow 0}\lim\limits_{k\rightarrow+\infty}\int_{\{|u_k|<r\}}f\left(u_{k}\right) u_{k}dx=0.$$ Notice that $\lim\limits_{k \rightarrow \infty} \int_{\mathbb{R}^{4}} f\left(u_{k}\right) u_{k}dx>0$, hence one of the following two cases must occur:

Case 1: $\lim\limits_{R\rightarrow+\infty}\lim\limits_{k \rightarrow \infty} \int_{\{|u_k|>R\}} f\left(u_{k}\right) u_{k} \mathrm{~d}x>0;$
\vskip0.1cm

Case 2: there exist some  $r,R>0$ such that $\lim_{k\rightarrow \infty}\int_{\{r<|u_k|<R\}} f\left(u_{k}\right) u_{k} \mathrm{~d}x>0$.
\medskip

If Case 1 occurs, since $f(t)$ has critical exponential growth when $t$ is large, then we have
$$\lim\limits_{R\rightarrow+\infty}\lim_{k \rightarrow \infty} \int_{\{|u_k|>R\}}\frac{f\left(t_{0} u_{k}\right)}{t_{0} u_{k}} u_{k}^{2} \mathrm{~d} x>\lim\limits_{R\rightarrow+\infty}\lim _{k \rightarrow \infty} \int_{\{|u_k|>R\}}f\left(u_{k}\right) u_{k} \mathrm{~d}x.$$
\medskip

If Case 2 occurs, since $u_{k}$ are bounded on the set $\{x|r\leq |u_k|\leq R\}$, then $|\{x|r\leq |u_k|\leq R\}|>\beta$ for some $\beta>0.$ Hence
$$\lim_{k \rightarrow \infty} \int_{\{r\leq |u_k|\leq R\}}\frac{f\left(t_{0} u_{k}\right)}{t_{0} u_{k}} u_{k}^{2} \mathrm{~d} x>\lim _{k \rightarrow \infty} \int_{\{r\leq |u_k|\leq R\}}f\left(u_{k}\right) u_{k} \mathrm{~d}x.$$

Combining the above estimates, the claim follows from the monotonicity of $\frac{f(t)}{t}$. Therefore, we derive that

\[
\underset{k\rightarrow\infty}{\lim}\int_{\mathbb{R}^{4}}\frac{f\left(
t_{k}u_{k}\right)  }{t_{k}u_{k}}u_{k}^{2}dx\geq\underset{k\rightarrow\infty} {\lim}  \int_{\mathbb{R}^{4}}%
\frac{f\left(  t_{0}u_{k}\right)  }{t_{0}u_{k}}u_{k}^{2}dx >\underset{k\rightarrow\infty}{\lim}\int_{\mathbb{R}%
^{4}}f\left(  u_{k}\right)  u_{k}dx>0
\]
which contradicts (\ref{converg}).
\vskip 0.1cm

Now, by\ (\ref{vanish}), we can write
\begin{align*}
m_{\infty}  &  \leq\lim_{k\rightarrow+\infty}I_{\infty}\left(  t_{k}%
u_{k}\right)  =\lim_{k\rightarrow+\infty}\left(  I_{V}\left(  t_{k}%
u_{k}\right)  +\frac{1}{2}t_{k}^{2}\int_{\mathbb{R}^{4}}(V_{\infty}
-V(x))\left\vert u_{k}\right\vert ^{2}dx\right) \\
&  =\lim_{k\rightarrow+\infty}I_{V}\left(  t_{k}u_{k}\right)  \leq
\lim_{k\rightarrow+\infty}t_{k}^{2}\left(  \int_{\mathbb{R}^{4}}\left(
\left\vert \nabla u_{k}\right\vert ^{2}+V\left(  x\right)  \left\vert
u_{k}\right\vert ^{2}\right)  dx-\int_{\mathbb{R}^{4}}\frac{F\left(
t_{k}u_{k}\right)  }{t_{k}^{2}u_{k}^{2}}u_{k}^{2}dx\right).
\end{align*}
This together with the monotonicity of $\frac{F\left(
t\right)  }{t^{2}}$ (see Remark \ref{adrem}) and $\lim\limits_{k\rightarrow \infty}t_k=1$ gives
\[
m_{\infty}\leq\lim_{k\rightarrow+\infty}I_{V}\left(  u_{k}\right)  =m_{V}%
\]
which contradicts (\ref{compare}). This accomplishes the proof of Lemma \ref{lem4.3}.
\end{proof}

Next, we claim that

\begin{lemma}\label{lem4.4}  It holds that
$$\lim_{k\rightarrow+\infty}\int_{\mathbb{R}^4}f(u_k)u_kdx=\lim_{L\rightarrow +\infty}\lim_{k\rightarrow+\infty}\int_{B_L(0)}f(u_k)u_kdx,\ \
\lim_{L\rightarrow +\infty}\lim_{k\rightarrow+\infty}\int_{\mathbb{R}^4\setminus B_L(0)}f(u_k)u_kdx=0.$$
\end{lemma}
\begin{proof}
Similar to the proof of Lemma \ref{lem3.4}, we define $M=\lim\limits_{k\rightarrow+\infty}\int_{\mathbb{R}^4}f(u_k)u_kdx$,$$ M^0=\lim\limits_{L\rightarrow +\infty}\lim\limits_{k\rightarrow+\infty}\int_{B_L(0)}f(u_k)u_kdx,\ \    M^{\infty}=\lim\limits_{L\rightarrow +\infty}\lim\limits_{k\rightarrow+\infty}\int_{\mathbb{R}^4\setminus B_L(0)}f(u_k)u_kdx.$$ We first show that
$(M^{0},M^{\infty})=(M,0)$ or $(M^{0},M^{\infty})=(0,M)$. Since $u_k\in \mathcal{N}_V$, then
\begin{equation}\begin{split}
&\int_{\mathbb{R}^4}(|\Delta u_k|^2+V(x)|u_k|^2)dx\\
&\ \ =\lim\limits_{L\rightarrow +\infty}\lim\limits_{k\rightarrow+\infty}\int_{B_L(0)}f(u_k)u_kdx+\lim\limits_{L\rightarrow +\infty}\lim\limits_{k\rightarrow+\infty}\int_{\mathbb{R}^4\setminus B_L(0)}f(u_k)u_kdx\\
&\ \ =M^{0}+M^{\infty}.
\end{split}\end{equation}
Noticing that we can write $\int_{\mathbb{R}^4}(|\Delta u_k|^2+V(x)|u_k|^2)dx$ as
\begin{align*}
  \int_{\mathbb{R}^4}(|\Delta u_k|^2+V(x)|u_k|^2)dx&=\lim\limits_{L\rightarrow +\infty}\lim\limits_{k\rightarrow+\infty}\int_{B_L(0)}(|\Delta u_k|^2+V(x)|u_k|^2)dx+\\&+\lim\limits_{L\rightarrow +\infty}\lim\limits_{k\rightarrow+\infty}\int_{\mathbb{R}^4\setminus B_L(0)}(|\Delta u_k|^2+V(x)|u_k|^2)dx.
\end{align*}

Hence we can assume that $$\lim\limits_{L\rightarrow +\infty}\lim\limits_{k\rightarrow+\infty}\int_{B_L(0)}(|\Delta u_k|^2+V(x)|u_k|^2)dx\leq \lim\limits_{L\rightarrow +\infty}\lim\limits_{k\rightarrow+\infty}\int_{B_L(0)}f(u_k)u_kdx$$ or $$\lim\limits_{L\rightarrow +\infty}\lim\limits_{k\rightarrow+\infty}\int_{\mathbb{R}^4\setminus B_L(0)}(|\Delta u_k|^2+V(x)|u_k|^2)dx\leq \lim\limits_{L\rightarrow +\infty}\lim\limits_{k\rightarrow+\infty}\int_{\mathbb{R}^4\setminus B_L(0)}f(u_k)u_kdx.$$
\vskip0.1cm

For a sufficiently large number $L>0$, define the  function $u_{k,L}^{*}=u_k\phi_{L}^{*}$ ($*=0$ or $\infty$) as in Lemma \ref{lem3.4}. We can easily verify that
$$\lim\limits_{L\rightarrow +\infty}\lim\limits_{k\rightarrow+\infty}\int_{\mathbb{R}^4}(|\Delta u_{k,L}^{0}|^2+V(x)|u_{k,L}^{0}|^2)dx=\lim\limits_{L\rightarrow +\infty}\lim\limits_{k\rightarrow+\infty}\int_{B_L(0)}(|\Delta u_k|^2+V(x)|u_k|^2)dx$$
$$\lim\limits_{L\rightarrow +\infty}\lim\limits_{k\rightarrow+\infty}\int_{\mathbb{R}^4}(|\Delta u_{k,L}^{\infty}|^2+V(x)|u_{k,L}^{\infty}|^2)dx=\lim\limits_{L\rightarrow +\infty}\lim\limits_{k\rightarrow+\infty}\int_{\mathbb{R}^4\setminus B_L(0)}(|\Delta u_k|^2+V(x)|u_k|^2)dx$$
$$\lim\limits_{L\rightarrow +\infty}\lim\limits_{k\rightarrow+\infty}\int_{\mathbb{R}^4}f(u_{k,L}^{0})u_{k,L}^0dx=\lim\limits_{L\rightarrow +\infty}\lim\limits_{k\rightarrow+\infty}\int_{B_L(0)}f(u_k)u_kdx$$
$$\lim\limits_{L\rightarrow +\infty}\lim\limits_{k\rightarrow+\infty}\int_{\mathbb{R}^4}f(u_{k,L}^{\infty})u_{k,L}^{\infty}dx=\lim\limits_{L\rightarrow +\infty}\lim\limits_{k\rightarrow+\infty}\int_{\mathbb{R}^4\setminus B_L(0)}f(u_k)u_kdx.$$
$$\lim\limits_{L\rightarrow +\infty}\lim\limits_{k\rightarrow+\infty}\int_{\mathbb{R}^4}F(u_{k,L}^{0})dx=\lim\limits_{L\rightarrow +\infty}\lim\limits_{k\rightarrow+\infty}\int_{B_L(0)}F(u_k)dx$$
$$\lim\limits_{L\rightarrow +\infty}\lim\limits_{k\rightarrow+\infty}\int_{\mathbb{R}^4}F(u_{k,L}^{\infty})dx=\lim\limits_{L\rightarrow +\infty}\lim\limits_{k\rightarrow+\infty}\int_{\mathbb{R}^4\setminus B_L(0)}F(u_k)dx.$$
Without loss of generality, we can assume that $$\lim\limits_{L\rightarrow +\infty}\lim\limits_{k\rightarrow+\infty}\int_{\mathbb{R}^4}(|\Delta u_{k,L}^0|^2+V(x)|u_{k,L}^0|^2)dx\leq \lim\limits_{L\rightarrow +\infty}\lim\limits_{k\rightarrow+\infty}\int_{\mathbb{R}^4}f(u_{k,L}^0)u_{k,L}^0dx,$$ then there exists $t_{k,L}^{0}$ such that $t_{k,L}^{0}u_{k,L}^0\in \mathcal{N}_V$. Obviously, $\lim\limits_{L\rightarrow +\infty}\lim\limits_{k\rightarrow+\infty}t_{k,L}^{0}\leq 1$. If
$t_{k,L}^{0}\leq 1$, then
\begin{equation}\label{eq3,1}\begin{split}
I_V(t_{k,L}^{0}u_{k,L}^0)&=\frac{1}{2}\int_{\mathbb{R}^4}\big(f(t_{k,L}^{0}u_{k,L}^0)t_{k,L}^{0}u_{k,L}^0-2F(t_{k,L}^{0}u_{k,L}^0)\big)dx\\
&\leq \frac{1}{2}\int_{\mathbb{R}^4}\big(f(u_{k,L}^0)u_{k,L}^0-2F(u_{k,L}^0)\big)dx.
\end{split}\end{equation}
If $t_{k,L}^{0}\geq 1$, then
\begin{equation}\label{eq3,2}\begin{split}
I_V(t_{k,L}^{0}u_{k,L}^0)&=\frac{1}{2}(t_{k,L}^{0})^2\int_{\mathbb{R}^4}(|\Delta u_{k,L}^{0}|^2+V(x)|u_{k,L}^{0}|^2)dx-\int_{\mathbb{R}^4}F(t_{k,L}^{0}u_{k,L}^0)dx\\
&\leq \frac{1}{2}(t_{k,L}^{0})^2\int_{\mathbb{R}^4}(|\Delta u_{k,L}^{0}|^2+V(x)|u_{k,L}^{0}|^2)dx-\int_{\mathbb{R}^4}F(u_{k,L}^0)dx\\
\end{split}\end{equation}
Combining the above estimate, we derive that
\begin{equation}\begin{split}
\lim\limits_{L\rightarrow +\infty}\lim\limits_{k\rightarrow+\infty}I(t_{k,L}^{0}u_{k,L}^0)&\leq \frac{1}{2}\lim\limits_{L\rightarrow +\infty}\lim\limits_{k\rightarrow+\infty}\int_{\mathbb{R}^4}\big(f(u_{k,L}^0)u_{k,L}^0-2F(u_{k,L}^0)\big)dx\\
&\ \ +\frac{1}{2}\lim\limits_{L\rightarrow +\infty}\lim\limits_{k\rightarrow+\infty}\int_{\mathbb{R}^4}\big(f(u_{k,L}^{\infty})u_{k,L}^{\infty}-2F(u_{k,L}^{\infty})\big)dx\\
&=\lim\limits_{k\rightarrow +\infty}I(u_k)=m_V.
\end{split}\end{equation}
On the other hand, from the definition of $m_V$, we know that $m_V\leq \lim\limits_{L\rightarrow +\infty}\lim\limits_{k\rightarrow+\infty}I(t_{k,L}^{0}u_{k,L}^0)$. Combining the above estimate, we conclude that $$\lim\limits_{L\rightarrow +\infty}\lim\limits_{k\rightarrow+\infty}\int_{\mathbb{R}^4}\big(f(u_{k,L}^{\infty})u_{k,L}^{\infty}-2F(u_{k,L}^{\infty})\big)dx=0,$$
that is $$\lim\limits_{L\rightarrow +\infty}\lim\limits_{k\rightarrow+\infty}\int_{\mathbb{R}^4\setminus B_{L}(0)}\big(f(u_{k})u_{k}-2F(u_{k})\big)dx=0,$$
which together with (A-R) condition implies that $M^{\infty}=\lim\limits_{L\rightarrow +\infty}\lim\limits_{k\rightarrow+\infty}\int_{\mathbb{R}^4\setminus B_L(0)}f(u_k)u_kdx=0$ and $(M^{0},M^{\infty})=(M,0)$. Similarly, we can prove that $(M^{0},M^{\infty})=(0,M)$ if we assume that
$$\lim\limits_{L\rightarrow +\infty}\lim\limits_{k\rightarrow+\infty}\int_{\mathbb{R}^4}(|\Delta u_{k,L}^{\infty}|^2+V(x)|u_{k,L}^{\infty}|^2)dx\leq \lim\limits_{L\rightarrow +\infty}\lim\limits_{k\rightarrow+\infty}\int_{\mathbb{R}^4}f(u_{k,L}^{\infty})u_{k,L}^{\infty}dx.$$
$(M^{0},M^{\infty})=(0,M)$ being impossible is a direct result of $u\neq 0$. This accomplishes the proof of Lemma \ref{lem4.4}.
\end{proof}

\begin{lemma}\label{lem3.5}
There holds $\lim\limits_{k\rightarrow +\infty}\int_{\mathbb{R}^4}F(u_k)dx=\int_{\mathbb{R}^4}F(u)dx$.
\end{lemma}
\begin{proof}
It follows from Lemma \ref{lem3.4} that $\lim\limits_{L\rightarrow +\infty}\lim\limits_{k\rightarrow +\infty}\int_{\mathbb{R}^{4}\setminus B_{L}(0)}f(u_k)u_kdx=0$, which together with the (A-R) condition implies that $\lim\limits_{L\rightarrow +\infty}\lim\limits_{k\rightarrow +\infty}\int_{\mathbb{R}^{4}\setminus B_{L}(0)}F(u_k)dx=0$. In order to derive the desired convergence, we only need to prove that
$$\lim\limits_{L\rightarrow +\infty}\lim\limits_{k\rightarrow +\infty}\int_{B_{L}(0)}F(u_k)dx=\int_{\mathbb{R}^4}F(u)dx.$$
Indeed, for any $s>0$, we have
\begin{equation}\begin{split}
& |\int_{B_{L}(0)}F\left(u_k\right)  dx-\int_{B_{L}(0)}F\left(  u\right)  dx| \\
&  \leq |\int_{B_{L}(0)\cap \{|u_k|<s\}}
F\left(  u\right) dx-\int_{B_{L}(0)\cap \{|u_k|<s\}}F\left(  u\right)  dx| \\
&  \ \ \ \ +\left\vert \int_{B_{L}(0)\cap \{|u_k|\geq s\}}F\left(u_{k}\right)dx-\int_{B_{L}(0)\cap \{|u_k|\geq s\}}F\left(  u\right)  dx\right\vert \\
&  =I_{k,R,s}+II_{k,R,s}.
\end{split}\end{equation}
A direct application of the dominated convergence theorem leads to
$I_{k,R,s}\rightarrow0$. For $II_{k,R,s}$, from the condition (iii), we have
$F\left(  s\right)  \leq c f\left(  s\right)$. Then it follows that
\begin{align*}
\int_{B_{L}^{0}\cap\left\{  {|u_{k}|}\geq s\right\}  }F\left(
u_{k}\right)  dx  &  \leq\frac{c}{s}\int_{\mathbb{R}^{4}\cap\left\{
{|u_{k}|}\geq s\right\}  }f\left(  u_{k}\right)  {u_{k}%
}dx\\
&  =\frac{c}{s}\int_{\mathbb{R}^{4}}f\left(  u_{k}\right)  {u_{k}%
}dx\rightarrow0,\text{ as }s\rightarrow\infty,
\end{align*}
where we have used the fact that $\int_{\mathbb{R}^{4}}f\left(  u_{k}\right)
{u_{k}}dx$ is bounded. Consequently, $II_{k,R,s}\rightarrow0$, and the lemma
is finished.
\vskip0.1cm
\end{proof}

\begin{lemma}\label{lem4.6}
Let $u_k$ be a bounded sequence in $W^{2,2}(\mathbb{R}^4)$ converging weakly and for almost every $x\in \mathbb{R}^4$ to non-zero $u$. Furthermore, we also assume that $\lim\limits_{k\rightarrow +\infty}I_{V}(u_k)<\frac{16\pi^2}{\alpha_0}$ and
$\int_{\mathbb{R}^4}\big(|\Delta u|^2+V(x)|u|^2\big)dx>\int_{\mathbb{R}^4}f(u)udx$, then
$$\lim_{k\rightarrow+\infty}\int_{\mathbb{R}^4}f(u_k)u_kdx=\int_{\mathbb{R}^4}f(u)udx.$$
\end{lemma}

\begin{proof}
According to Lemma \ref{lem4.4}, we only need to prove that $$\lim_{L\rightarrow +\infty}\lim_{k\rightarrow+\infty}\int_{B_L(0)}f(u_k)u_kdx=
\int_{\mathbb{R}^4}f(u)udx.$$
It follows the lower semicontinuity of the norm in $W^{2,2}(\mathbb{R}^4)$ that
that $$\lim\limits_{k\rightarrow\infty}\int_{\mathbb{R}^4}\big(|\Delta u_k|^2+V(x)|u_k|^2\big)dx\geq \int_{\mathbb{R}^4}\big(|\Delta u|^2+V(x)|u|^2\big)dx.$$
We divide the proof into the following case.
\vskip0.1cm

Case 1: $\int_{\mathbb{R}^4}\big(|\Delta u_k|^2+V(x)|u_k|^2\big)dx=\int_{\mathbb{R}^4}\big(|\Delta u|^2+V(x)|u|^2\big)dx$, then according to convexity of the norm and the equivalence of norms, we see that
$u_k\rightarrow u$ in $W^{2,2}(\mathbb{R}^4)$, hence $u_k\rightarrow u$ in $L^{p}(\mathbb{R}^4)$ for any $p\geq 2$. Hence
it follows from Adams inequality in $W^{2,2}(\mathbb{R}^4)$ that for any $p_0>1$, $\sup_{k}\int_{\mathbb{R}^4} \big(f(u_k)u_k\big)^{p_0}dx<\infty$,
which implies that
\begin{equation}\label{con1}
\lim_{L\rightarrow+\infty}\lim_{k\rightarrow\infty}\int_{B_{L}}f(u_k)u_kdx=\int_{\mathbb{R}^4}f(u)udx.
\end{equation}

Case 2: If $\lim\limits_{k\rightarrow\infty}\int_{\mathbb{R}^4}\big(|\Delta u_k|^2+V(x)|u_k|^2\big)dx>\int_{\mathbb{R}^4}\big(|\Delta u|^2+V(x)|u|^2\big)dx$, we set
 $$ v_k:=\frac{u_k}{\lim\limits_{k\rightarrow\infty} \|u_k\|_{W_V^{2,2}(\mathbb{R}^4)}}\ \mbox{and}\ v_0:=\frac{u}{\lim\limits_{k\rightarrow\infty}\|u_k\|_{{W}_V^{2,2}(\mathbb{R}^4)}}.$$

 We claim there exists $q_0>1$ sufficiently $1$ such that
 \begin{equation}\label{d.7}
  q_0\|u_k\|^2_{W_V^{2,2}(\mathbb{R}^4)}<\frac{32\pi^2}{1-\|v_0\|^2_{{W}_V^{2,2}(\mathbb{R}^4)}}.
  \end{equation}
 Indeed, we can apply the condition (i) and (ii) to obtain
  \begin{equation}\begin{split}\label{d.8}
  &\lim\limits_{k\rightarrow\infty}\|u_k\|^2_{W_V^{2,2}(\mathbb{R}^4)}\big(1-\|v_0\|^2_{{W}_V^{2,2}(\mathbb{R}^4)}\big)\\
  &\ \ =\lim\limits_{k\rightarrow\infty}\|u_k\|^2_{W_V^{2,2}(\mathbb{R}^4)}\Big(1-\frac{\|u\|^2_{W_V^{2,2}(\mathbb{R}^4)}}{\|u_k\|^2_{W_V^{2,2}(\mathbb{R}^4)}}\Big)\\
 &\ \ =2\lim\limits_{k\rightarrow+\infty}I_{V}(u_k)+2\int_{\mathbb{R}^4}F(u_k)dx-2I_V(u)-2\int_{\mathbb{R}^4}F(u)dx\\
&\ \ <\frac{32\pi^2}{\alpha_0}.
\end{split}\end{equation}
Combining the above estimate with the concentration compactness principle for the Adams inequality which was established in \cite{chenluzhang}
in $W^{2,2}(\mathbb{R}^4)$, one can derive that there exists $p_0>1$ such that
\begin{eqnarray}\label{d.9}
\sup_{k}\int_{\mathbb{R}^4}\big(f(u_k)u_k\big)^{p_0}dx<\infty.
\end{eqnarray}
Then it follows from the Vitali convergence theorem that $$\lim_{L\rightarrow+\infty}\lim_{k\rightarrow\infty}\int_{B_{L}}f(u_k)u_kdx=\int_{\mathbb{R}^4}f(u)udx,$$
which implies the proof of Lemma \ref{lem4.6}.
\end{proof}

Now we are in position to show the existence of ground-state solutions to equation (\ref{degen}) is $V_\infty<\gamma^{*}$.

\begin{proof} [Proof of Theorem \ref{thm3}]

We will prove that if $V_\infty<\gamma^{*}$, then $m_{V}$ is achieved by some $u$. We claim that $$\int_{\mathbb{R}^4}\big(|\Delta u|^2+V(x)|u|^2\big)dx\leq \int_{\mathbb{R}^4}f(u)udx.$$
Suppose this is false, then
\begin{equation}\label{t1}
\int_{\mathbb{R}^4}\big(|\Delta u|^2+V(x)|u|^2\big)dx>\int_{\mathbb{R}^4}f(u)udx
\end{equation}
In view of Lemma \ref{impor} and Lemma \ref{lem4.6}, we derive that  $$\lim_{k\rightarrow \infty}\int_{\mathbb{R}^4}f(u_k)u_kdx=\int_{\mathbb{R}^4}f(u)udx.$$
This implies that
\begin{equation}\begin{split}
\int_{\mathbb{R}^4}\big(|\Delta u|^2+V(x)|u|^2\big)dx&\leq \lim_{k\rightarrow \infty}\int_{\mathbb{R}^4}\big(|\Delta u_k|^2+V(x)|u_k|^2\big)dx\\
&=\lim_{k\rightarrow \infty}\int_{\mathbb{R}^4}f(u_k)u_kdx=\int_{\mathbb{R}^4}f(u)udx \\ &<\int_{\mathbb{R}^4}\big(|\Delta u|^2+V(x)|u|^2\big)dx,
\end{split}\end{equation}
which is a contradiction. This proves the claim.
\vskip 0.1cm
Since $$\int_{\mathbb{R}^4}\big(|\Delta u|^2+V(x)|u|^2\big)dx\leq \int_{\mathbb{R}^4}f(u)udx,$$
there exists $\gamma_{0}\in (0,1]$ such that $\gamma_{0}u\in \mathcal{N}_{V}$. According to the definition of  $m_{V}$, we derive that
\begin{equation}\begin{split}
m_V\leq I_V(\gamma_{0}u)&=\frac{1}{2}\int_{\mathbb{R}^4}\big(f(\gamma_{0}u)(\gamma_{0}u)-2F(\gamma_{0}u)\big)dx\\
&\leq \frac{1}{2}\int_{\mathbb{R}^4}\big(f(u)(u)-2F(u)\big)dx\\
&\leq \lim_{k\rightarrow \infty}\frac{1}{2}\int_{\mathbb{R}^4}\big(f(u_k)(u_k)-2F(u_k)\big)dx\\
&=\lim_{k\rightarrow \infty}I_{V}(u_k)=m_V.
\end{split}\end{equation}
This implies that $\gamma_0=1$ and $u\in \mathcal{N}_{V}$ and $I_{V}(u)=m_V$. We accomplish the proof of Theorem \ref{thm3}.
\end{proof}

\end{document}